\documentclass[12pt]{article}
\date{\today}
\usepackage[usenames,dvipsnames]{color}
\usepackage{amsmath,amsthm,amsfonts,amssymb,multicol,amscd,amsbsy}
\usepackage{graphicx}
\usepackage{fullpage}
\usepackage{booktabs}
\usepackage{array}
\usepackage{subfigure}
\usepackage{enumerate}
\usepackage{url}
\usepackage[nodayofweek]{datetime}
\usepackage[english]{babel}
\usepackage{mathtools}
\usepackage[toc,page]{appendix}
\usepackage{verbatim} 
\usepackage{amsthm} 
\usepackage{enumitem}
\usepackage{enumerate}
\usepackage{bbm} 
\usepackage[normalem]{ulem} 
\usepackage{bm}
\usepackage{hyperref}
\usepackage{latexsym}
\usepackage[left=4cm, right=4cm, top=2.5cm, bottom=2.5cm]{geometry}

\DeclarePairedDelimiter\ceil{\lceil}{\rceil}

\usepackage{authblk}

\newcommand{\be}{\begin}
\newcommand{\e}{\end}
\newcommand{\beq}{\begin{equation}}
\newcommand{\eeq}{\end{equation}}
\newcommand{\beqs}{\begin{equation*}}
\newcommand{\eeqs}{\end{equation*}}

\newcommand{\ol}{\overline}
\newcommand{\ul}{\underline}

\renewcommand{\l}{\left}
\renewcommand{\r}{\right}
\renewcommand{\d}{\mathrm{d}} 
\newcommand{\De}{\Delta}

\newcommand{\inn}[1]{\langle #1 \rangle}
\newcommand{\set}[1]{\mathbb{#1}}
\newcommand{\curly}[1]{\mathcal{#1}}

\newcommand{\setof}[2]{\left\{ #1\; : \;#2 \right\}}

\newcommand{\R}{\set{R}}
\newcommand{\C}{\set{C}}
\newcommand{\Z}{\set{Z}}
\newcommand{\E}{\set{E}}
\newcommand{\T}{\set{T}}

\newcommand{\hi}{\mathcal{H}}

\newcommand{\om}{\omega}

\newcommand{\eps}{\epsilon}

\newcommand{\sig}{\sigma}
\newcommand{\al}{\alpha}
\newcommand{\de}{\delta}

\newcommand{\vp}{\varphi}

\newcommand{\ind}{\mathbbm{1}}		

\newcommand{\ttmatrix}[4]{\left(\be{array}{cc} #1&#2\\	#3&#4 \e{array}	\right)}

\newcommand{\qmexp}[1]{\langle #1\rangle}
\newcommand{\jap}[1]{\langle #1\rangle}

\newcommand{\abs}[1]{\left\vert#1\right\vert}
\renewcommand{\it}{\infty}

\newcommand{\supp}{\,\mathrm{supp}\,}

\newcommand{\ux}{\underline{x}}

\newcommand{\norm}[1]{ || #1 ||}
\newcommand{\br}[1]{\langle #1 \rangle}

\newcommand{\ft}[1]{\widehat{{#1}}}			

\newtheorem{thm}{Theorem}[section]
\newtheorem{lm}[thm]{Lemma}
\newtheorem{cor}[thm]{Corollary}
\newtheorem{prop}[thm]{Proposition}

\theoremstyle{definition}
\newtheorem{defn}[thm]{Definition}

\newtheorem{exam}[thm]{Example}

\numberwithin{equation}{section}

\theoremstyle{remark}

\def\dotuline{\bgroup
  \ifdim\ULdepth=\maxdimen  
   \settodepth\ULdepth{(j}\advance\ULdepth.4pt\fi
  \markoverwith{\begingroup
  \advance\ULdepth0.08ex
  \lower\ULdepth\hbox{\kern.15em .\kern.1em}%
  \endgroup}\ULon}

\def\dashuline{\bgroup
  \ifdim\ULdepth=\maxdimen  
   \settodepth\ULdepth{(j}\advance\ULdepth.4pt\fi
  \markoverwith{\kern.15em
  \vtop{\kern\ULdepth \hrule width .3em}%
  \kern.15em}\ULon}
\allowdisplaybreaks

\begin{document}
\title{On the averaged Green's function of an elliptic equation with random coefficients}
\author[1]{Jongchon Kim}
\author[1,2]{Marius Lemm}
\affil[1]{\small{School of Mathematics, Institute for Advanced Study}}
\affil[2]{\small{Department of Mathematics,  Harvard University}}
\maketitle

\begin{abstract}
We consider a divergence-form elliptic difference operator on the lattice $\Z^d$, with a coefficient matrix that is an i.i.d.\ perturbation of the identity matrix. Recently, Bourgain introduced novel techniques from harmonic analysis to prove the convergence of the Feshbach-Schur perturbation series related to the averaged Green's function of this model.
 Our main contribution is a refinement of Bourgain's approach which improves the key decay rate from $-2d+\eps$ to $-3d+\eps$. (The optimal decay rate is conjectured to be $-3d$.) As an application, we derive estimates on higher derivatives
 of the averaged Green's function which go beyond the second derivatives considered by Delmotte-Deuschel and related works.
\end{abstract}

\section{Introduction}

In the late 1950s, De Giorgi, Nash and Moser \cite{deGiorgi,Nash,Moser} completed the classical regularity theory for elliptic and parabolic equations with bounded and measurable coefficients. Their results include the H\"older regularity of weak solutions $u$ to the divergence-form elliptic equation $\nabla^* \mathbf{A}(x) \nabla u=0$ with rough coefficient matrix $\mathbf{A}(x)$. Subsequently, it was also shown that the Green's function $G_{\mathbf{A}}(x,y)$ is controlled by the Green's function of the ordinary Laplacian. Specifically, when $d\geq 3$, it holds that
\beq\label{eq:Aronson}
|G_{\mathbf{A}}(x,y)|\leq C |x-y|^{2-d},
\eeq
for all $x,y\in \R^d$; see \cite{LSW,A1,A2}. 

When the coefficient matrix $\mathbf{A}(x)$ is generated by a stationary \emph{random} process, one may consider regularity properties that hold on average or with high probability; see, e.g., \cite{AM,GNOreg,GloriaOtto,MO1}.  
Here we focus on the \emph{averaged} (or ``annealed'') Green's function $\mathbb E[G_{\mathbf{A}}(x,y)]$, which is translation-invariant in the sense that 
\[ \mathbb E[G_{\mathbf{A}}(x,y)]=G(x-y)\]
for some function $G$, \emph{cf}. \eqref{eq:Gr}. In this setting, Conlon-Naddaf \cite{CN} (see also \cite{Con}) observed that the averaged Green's function $G(x)$ is continuously differentiable for $x\neq 0$ and its derivative satisfies the decay estimate
$$
|\nabla G(x)|\leq C|x|^{1-d},
$$
when working on either $\R^d$ or $\Z^d$ with $d\geq 3$. Note that the decay rate $1-d$ is optimal in view of the Green's function of the ordinary Laplacian. In the discrete setting, \cite{CN} also proved that the second derivatives of $G$ are controlled by $C_\de (1+|x|)^{-d+\de}$ for arbitrarily small $\delta>0$. 

Their result was extended by Delmotte-Deuschel \cite{DD} who adapted the classical regularity theory to the random setting. They showed that the second derivatives of the averaged Green's function can actually be controlled with the optimal decay rate:
\beq\label{eq:DD}
|\nabla^{\al} G(x)|\leq C|x|^{2-d-|\al|},\qquad \textnormal{for any multi-index } |\al|\leq 2,
\eeq
again on $\R^d$ or $\Z^d$ when $d\geq 3$. In fact, \cite{DD} establish a stronger version of \eqref{eq:DD} where one takes the absolute value before taking expectation. Moreover, they have a similar result for $d=2$, i.e., \eqref{eq:DD} holds with $1\leq|\al|\leq 2$, if the first and second derivatives of $G$ are properly interpreted. In the discrete case (i.e., on $\Z^d$), there is no singularity near the origin and so $|x|^{2-d-|\al|}$ can be replaced by $(1+|x|)^{2-d-|\al|}$ in \eqref{eq:DD}. 

We mention that the elliptic results presented here have parabolic analogs; see, e.g., \cite{CKS,CN,DD,Nash}. 

In the last few years, the derivative estimate \eqref{eq:DD} on the averaged Green's function has been generalized to higher moments and to the non-scalar case \cite{CGO,GM,MO1,MO2}. One reason for the continued interest in these Green's function estimates is that they have applications to the quantitative theory of \emph{stochastic homogenization}. Consider for example a family of equations of the form
\beq\label{eq:hom}
\nabla^* \mathbf{A}\l(\frac{x}{\eps}\r) \nabla u_\eps=0,
\eeq
indexed by $\eps>0$, with a random coefficient matrix $\mathbf{A}(x)$. Under certain assumptions on $\mathbf{A}(x)$, it is known that, as $\eps\to 0$, a solution $u_\eps$ to \eqref{eq:hom} can be approximated by a solution $u$ to a ``homogenized'' deterministic \emph{constant coefficient} equation. This general phenomenon is called stochastic homogenization and has been extensively studied; see \cite{JKO,Kozlov,SpencerNaddaf,PV,Yu} and the more recent works \cite{AKM,AM,CN2,CS,DGO,GNO,GloriaOtto,GloriaOtto1,GloriaOtto2}. While stochastic homogenization furnishes part of our general motivation, we will not directly discuss it anymore in the following.\\

Despite the recent research activities on the averaged Green's function, it has been unknown, to the best of our knowledge, whether the optimal decay rate in \eqref{eq:DD} holds true beyond the second derivatives. A consequence of our results is that the estimate \eqref{eq:DD} indeed extends to higher order derivatives for all $|\alpha|\leq d+1$, in the discrete setting when $\mathbf{A}(x)$ is an i.i.d.\ perturbation of the identity matrix.

 Our argument is different from those in \cite{CN,DD} and is based on the line of research recently initiated by I.M.~Sigal \cite{Sigal} and J.~Bourgain \cite{B}. Bourgain gave a rather precise description of an averaged operator $\curly L$ (whose Green's function is exactly the averaged Green's function $G$ from above), by establishing the convergence of the Feshbach-Schur perturbation series. Our main result improves a key decay estimate for $\curly L$ obtained in \cite{B}; see Theorem \ref{thm:main} below. The estimate on the higher derivatives of the averaged Green's function is a corollary of this main result and is obtained by using standard tools from Fourier analysis.




We organize this paper as follows. In the remainder of this section, we give precise statements of our setup and main results and an outline of the argument. In Section \ref{sec:pre}, we provide background: (a) We precise operator-theoretic aspects of the setup, and (b) we recall two key tools introduced in \cite{B} and state abstract versions to be used later on. We prove our main results, Theorems \ref{thm:main} and \ref{thm:nlogn}, in Sections \ref{sec:mainresult} and \ref{sec:main2}, respectively. We prove the new derivative estimates on the averaged Green's function, Corollary \ref{cor:GF}, in the Appendix. In addition, we give proofs of the statements in Section \ref{sec:pre} in the Appendix for completeness. 
\\

\textbf{Notations}. Let $-\Delta = \nabla^* \nabla$ be the standard Laplacian on $\Z^d$, where $\nabla=(\nabla_1, \nabla_2,\ldots,\nabla_d)^T$ is the discrete derivative. For a function $u:\Z^d\to \R$ or $\C$, it is defined by
$$
\nabla_j u (x) := u(x+e_j)-u(x)
$$
for the $j$-th standard unit vector $e_j$. We denote by $\nabla^*=(\nabla_1^*, \ldots,\nabla_d^*)$ the adjoint of $\nabla$, where
$$
\nabla_j^* u (x) := u(x-e_j)-u(x).
$$

For a given multi-index $\alpha=(\alpha_1,\cdots,\alpha_d)\in \Z^d$, $\alpha_j\geq 0$, we write $|\alpha| = \sum_{j=1}^d \alpha_j$ and $\nabla^\alpha=\nabla_1^{\alpha_1} \cdots \nabla_d^{\alpha_d}$. 

\subsection{Statement of main results}

We continue with the precise setup of the model. Let $d\geq 2$. For each $x\in \Z^d$ and $\om\in \Omega$, an underlying probability space, let $\mathbf{A}(x) = \mathbf{A}(x,\om)$ be an i.i.d.\ perturbation of the identity matrix $\mathbf{I}_d$, i.e.,
\beq\label{eq:choice}
\mathbf{A}(x,\om):=(1+\de\sigma(x,\om))\mathbf{I}_d,
\eeq
where $0<\de<1$ is a small parameter and $\{ \sigma(x,\cdot) \}_{x\in \Z^d}$ is a bounded family of real-valued i.i.d.\ random variables which we normalize by the condition $\norm{\sigma}_{L^\infty(\Z^d\times \Omega)}\leq 1$. (Our results extend to non-diagonal i.i.d.\ perturbations; see Remark \ref{rm:gen}.) The i.i.d.\ model \eqref{eq:choice} is perhaps the simplest possible choice for the random coefficient matrix. See \cite[Theorem 1.2]{CN2}, where a corresponding homogenization problem is addressed. 

Let $L$ be the corresponding i.i.d.\ perturbation of the Laplacian on $\Z^d$, i.e.,
\begin{equation}\label{eqn:L}
L := \nabla^*\mathbf{A} \nabla = -\Delta + \delta \nabla^* \sigma \mathbf{I}_d \nabla.
\end{equation}
We also write $L_\omega := \nabla^*\mathbf{A}(\cdot,\omega) \nabla$ to emphasize the dependence on $\omega\in \Omega$. The \emph{main object of our study} is the averaged operator $\mathcal{L}$ defined by
\begin{equation}\label{eq:curlL}
\mathcal{L} := \l(\mathbb E \l[ L_{\om} ^{-1} \r] \r)^{-1}
\end{equation}
for $d\geq 3$. Here the inverses of $L_{\om}$ and of $\mathbb E \l[ L_{\om} ^{-1} \r]$ are well-defined for any $0<\de<1$ as maps between appropriate function spaces; see Sections \ref{sec:spaces} and \ref{sec:derive} for the details.

The averaged operator $\curly{L}$ governs the average behavior of solutions: if $u_\om$ is the solution to $L_\om u_\om = f$, then $\mathcal{L}[\mathbb Eu]=f$. In terms of the Green's function, the Green's function for $\curly{L}$ is equal to the averaged Green's function $\E [G_{\mathbf{A}}(x,y)]$ discussed earlier. To make this precise, recall that, given any $y\in \Z^d$, the Green's function $G_{\mathbf{A(\cdot,\om)}}(x,y)$ for $L_\om$ is the unique solution $g_\om^y $ (in an appropriate $\ell^p$ space) to the equation $L_{\om} g^y_{\om} = \delta_y$. Similarly, let $G(x-y)$ be the Green's function for the translation invariant operator $\curly L$ characterized by $\mathcal{L} G = \delta_0$. Then we have
\begin{equation}\label{eq:Gr}
\E[G_{\mathbf{A(\cdot,\om)}}(x,y)] = \E [L_\om^{-1} \delta_y (x)] = \mathcal{L}^{-1} \delta_y(x) = G(x-y),
\end{equation}
where $L_\om^{-1} \delta_y$ and $\mathcal{L}^{-1} \delta_y$ are some $\ell^p$ functions. Further explanations are deferred to Section \ref{sec:spaces}.\\

By introducing several novel techniques from harmonic analysis to the problem, Bourgain \cite{B} recently established the remarkable result that the operator $\mathcal{L}$ can be expressed as a convergent perturbation series for sufficiently small $\de>0$, and it admits the representation 
\beq\label{eq:Bexpansion}
\curly{L}=(1+\delta \E\sigma)(-\Delta) +  \nabla^* \mathbf{K}^{\delta} \nabla.
\eeq
Here, $\E\sigma\in \R$ is the expectation of any copy of $\sigma(x,\cdot)$ and $\mathbf{K}^{\delta}=( K_{i,j}^{\delta} )_{1\leq i,j\leq d}$ is an operator-valued matrix whose matrix elements $K_{i,j}^{\delta}$ are convolution operators on $\Z^d$. Bourgain also proved that the convolution kernel $K_{i,j}^{\delta}(x)$ is controlled by the decaying function $(1+|x|)^{-2d+\eps}$, with $\eps>0$ depending on $\de$. This implies that its Fourier transform $\widehat{K^\delta_{i,j}}$ belongs to the H\"{o}lder space $C^{d-1,1-\eps}(\T^d)$. 

The main result of this paper is the following improved decay estimate for the convolution kernel.

\begin{thm}[Main result]\label{thm:main} Let $d\geq 3$ and $0<\epsilon<1$. There is a constant $c_d>0$ such that the representation \eqref{eq:Bexpansion} is valid for any $0<\delta<c_d \epsilon$. Moreover, the convolution kernel of $K^{\delta}_{i,j}$ obeys the decay estimate
\beq\label{eq:main}
 |K^{\delta}_{i,j}(x-y)| \leq C_d \delta^2 (1+|x-y|)^{-3d+\epsilon}
 \eeq
for all $x,y\in \Z^d$, with an additional factor of $\delta^2$ when $x\neq y$. 

Consequently, the Fourier transform $\ft{K^{\delta}_{i,j}}$ is an element of the H\"older space $C^{2d-1,1-\eps}(\T^d)$; in particular, it has $2d-1$ continuous derivatives.
\end{thm}

\be{rmk}\label{rmk:main}
\be{itemize}
\item[(i)] The exponent $-3d+\epsilon$ in \eqref{eq:main} improves the exponent $-2d+\epsilon$ obtained in \cite{B}. Theorem \ref{thm:main} also yields additional factors of $\de$ and quantifies the dependence on $\eps$ for the allowed range of $\delta$. However, it is an interesting open question whether this dependence on $\eps$ can be completely removed.

\item[(ii)] Our work is motivated by a conjecture of Tom Spencer (private communication), which says that $-3d$ should be the optimal decay rate in \eqref{eq:main}. Note that our bound \eqref{eq:main} establishes the conjecture up to an arbitarily small $\eps>0$. The conjecture is supported by an examination of the $n=3$ term in the perturbation series \eqref{eq:series}, which is the leading contribution in $\de$ when $x\neq y$.

\e{itemize}
\e{rmk}


Our proof yields a similar result for some regularized versions of $L_{\om}$, with bounds that are uniform in the regularizing parameter and in this case one can include $d=2$. For instance, define the operator $L_{\mu,\om} := L_\om + \mu I$ for each $\mu>0$. It is strictly positive and therefore invertible on $\ell^2(\Z^d)$. We refer the reader to \cite[Lemma 4]{GloriaOtto1} for a pointwise decay estimate for the Green's function of $L_{\mu,\om}$. We state a version of Theorem \ref{thm:main} for the averaged operator $\curly{L}_\mu := \l(\mathbb E \l[ L_{\mu,\om}^{-1} \r] \r)^{-1}$. Note that the Green's function $G_\mu(x-y)$ for the operator $\curly{L_\mu}$ is the averaged Green's function associated with the operators $L_{\mu,\om}$; \emph{cf.} \eqref{eq:Gr}. 

\begin{thm}\label{thm:mainreg} Let $d\geq 2$, $\mu>0$ and $0<\epsilon<1$. There is a constant $c_d>0$ such that, any $0<\delta<c_d \epsilon$, we may write 
\beq\label{eq:Bexpansion2}
\curly{L}_\mu =   (1+\delta \E\sigma)(-\Delta) +\mu I +  \nabla^* \mathbf{K}_\mu^{\delta} \nabla
\eeq
for a convolution operator $\mathbf{K}_\mu^{\delta}= (K^{\delta}_{\mu,i,j})_{ 1\leq i,j \leq d}$. The convolution kernel of $K^{\delta}_{\mu,i,j}$ obeys the decay estimate
\beq\label{eq:main2}
 |K^{\delta}_{\mu,i,j}(x-y)| \leq C_d \delta^2 (1+|x-y|)^{-3d+\epsilon}
 \eeq
uniformly in $\mu>0$ for all $x,y\in \Z^d$, with an additional factor of $\delta^2$ when $x\neq y$. Consequently, the Fourier transform $\ft{K^{\delta}_{\mu,i,j}}$ is an element of the H\"older space $C^{2d-1,1-\eps}(\T^d)$.
\end{thm}

We omit the proof of Theorem \ref{thm:mainreg} since it is identical to the proof of Theorem \ref{thm:main} modulo the replacement of the positive operator $-\Delta$ by the strictly positive operator $-\Delta_\mu := -\Delta + \mu I$ in each step of the proof.

\be{rmk}\label{rm:gen}
It is straightforward to generalize Theorems \ref{thm:main} and \ref{thm:mainreg} for the coefficient matrix of the form 
\beq\label{eq:choice2}
\mathbf{A}(x,\om):=\mathbf{A}_0 + \de \mathbf{\Sigma}(x,\om).
\eeq
Here, $\mathbf{A}_0$ is a positive definite matrix satisfying
\[ \sum_{1\leq i,j\leq d}  a_i (\mathbf{A}_0)_{i,j} \overline{a_j} \geq c \sum_{ 1\leq i\leq d} |a_i|^2 \]
for some constant $c>0$ for any $a_i \in \C$ and $\mathbf{\Sigma}(x,\om) = (\sigma_{i,j}(x,\om))_{i,j}$ is a symmetric $d\times d$ matrix and \[ \setof{\sigma_{i,j}(x,\cdot)}{1\leq i,j\leq d,\,\, x\in \Z^d}\] is a family of identically distributed random variables satisfying the following independence condition: for any $1\leq i,j,i',j'\leq d$, the random variables $\sigma_{i,j}(x,\cdot)$ and $\sigma_{i',j'}(y,\cdot)$ are independent if $x\neq y$. (If complex-valued functions $u:\Z^d\to \C$ are considered, then it is necessary to restrict to Hermitian matrices $\Sigma$, in order to ensure ellipticity.)
In this setting, one finds
$$
\curly{L}= \nabla^* \E \mathbf{A} \nabla  +  \nabla^* \mathbf{K}^{\delta} \nabla
$$
and the operator kernel of $\mathbf{K}^{\delta}$ satisfies the bound \eqref{eq:main}.
\e{rmk}

As a corollary of Theorems \ref{thm:main} and \ref{thm:mainreg}, we establish decay estimates for the discrete derivatives of the averaged Green's functions $G(x-y)$ and $G_\mu(x-y)$ for the operators $L_{\om}$ and $L_{\mu,\om}$, respectively. These estimates extend the result \eqref{eq:DD} from \cite{CN, DD} to higher order derivatives for our choice of random environment. 

\be{cor}[Bounds on the averaged Green's function]\label{cor:GF}
There is a constant $c_d>0$ such that the following holds for any $0<\delta<c_d$.
\be{enumerate}[label=(\roman*)]
\item 
If $d\geq 3$, then
\[ |\nabla^\alpha G(x)| \leq C_\alpha (1+|x|)^{-(d-2+|\alpha|)} \]
holds for any multi-index $0\leq |\alpha | \leq d+1$. Similar estimates hold for $\nabla^\alpha G_\mu$ uniformly in $\mu>0$. 
\item
If $d= 2$, then 
\[ |\nabla^\alpha G_\mu(x)| \leq C_\alpha (1+|x|)^{-|\alpha|} \]
uniformly in $\mu >0$ for any multi-index $ 1 \leq |\alpha | \leq 3$.
\e{enumerate}
\e{cor}

As a consequence of Theorem \ref{thm:main} and Corollary \ref{cor:GF}, we obtain an estimate on the derivatives of the averaged solution $\mathbb{E} [u_\om]$ to $L_\om u_\om = f$.
\begin{cor} \label{cor:GF2} Let $d\geq 3$ and let $p_d$ be the Hardy-Littlewood-Sobolev exponent, i.e., $p_d^{-1}=2^{-1}+d^{-1}$. Assume that $f\in \ell^{p_d}(\Z^d)$. For each $\om\in \Omega$, there exists a unique solution $u_\om \in \ell^{q_d}(\Z^d)$, $q_d^{-1} := 2^{-1}-d^{-1}=p_d'$, to \[L_\om u_\om = f \] such that $|\nabla u_\om| \in \ell^2(\Z^d)$. The averaged solution can be represented by
\[ \mathbb{E} [u_\om] = \curly{L}^{-1} f = G*f. \]
Moreover, there is a constant $c_d$ such that for any $0<\delta<c_d$, the derivatives of the average can be estimated pointwise by
\[ |\nabla^{\alpha} \mathbb{E} [u_\om](x) | \leq C \sum_{y\in \Z^d} \frac{|f(y)|}{(1+|x-y|)^{d-2+|\alpha|}}. \]
for any multi-index $0\leq |\alpha | \leq d+1$.
\end{cor}

We prove Corollaries \ref{cor:GF} and \ref{cor:GF2} in the Appendix. For the former, we use that the Fourier space representations of $\nabla^\alpha G$ and $\nabla^\alpha G_\mu$ can be controlled via Theorems \ref{thm:main} and \ref{thm:mainreg}.


In the next subsection, we sketch the proof of Theorem \ref{thm:main}. This ultimately motivates an alternative approach towards Theorem \ref{thm:main}. We raise a question regarding that approach and partially answer the question by our second main result, Theorem \ref{thm:nlogn}.

In the following, we commonly abuse notation and identify operators with their kernels, i.e., we do not distinguish notationally between a function $K: \Z^d \times \Z^d \to \C$ and the operator $Kf(x) = \sum_{y\in \Z^d} K(x,y) f(y)$. From now on, $C$ denotes a positive constant that is uniform in all the parameters except dimension and whose numerical value may change from line to line.


\subsection{Outline of the argument} 
 \label{sect:summary}
Our proof of Theorem \ref{thm:main} relies on the techniques introduced in \cite{Sigal,B}. We start by introducing Bourgain's approach and then briefly explain how we can refine that argument.

The starting point is the Feshbach-Schur map that yields the perturbation series representation for $\mathbf{K}^\de$
\beq\label{eq:series}
\mathbf{K}^\de=\delta \sum_{n=1}^\it (-\de)^n P \sigma (\mathbf{K} P^\perp \sigma)^n,
\eeq
where $P=\mathbb E[\cdot]$ and $P^\perp=I-P$ and we introduced the operator-valued matrix $\mathbf{K}:=\nabla (\De)^{-1}\nabla^*$. Here we emphasize that by the operator $P^\perp \sigma$ we mean the \emph{composition} of $P^\perp$ and the multiplication operator associated with $\sigma \mathbf{I}_d$. See \cite{B} (and also Section \ref{sec:derive}) for the derivation of \eqref{eq:series}.

Note that each entry of the matrix $\mathbf K$ is a \emph{singular integral operator} of convolution type. However, the reader is invited to think of $\mathbf{K}$ as a usual \emph{scalar} singular integral operator acting on $\Z^d$. See the beginning of Section \ref{sec:bourgain} for a discussion on such operators. 

Our key result, Proposition \ref{prop:key}, says that 
\begin{equation}\label{eqn:key}
|P \sigma (\mathbf{K} P^\perp \sigma)^n(x,y)| \leq \eps^3 \l(\frac{C}{\eps}\r)^n \jap{x-y}^{-3d+\eps} ,
\end{equation}
where we write $\langle\cdot\rangle=1+|\cdot|$. This shows that the series in \eqref{eq:series} is convergent for sufficiently small $\delta>0$ and, therefore, implies Theorem \ref{thm:main}. 

To show \eqref{eqn:key}, one writes the kernel of $P \sigma (\mathbf{K} P^\perp \sigma)^n$ as
\beq\label{eq:tobound}
\begin{aligned}
&P \sigma (\mathbf{K} P^\perp \sigma)^n(x_0,x_n)\\
=&P\sum\limits_{\ul{x}\in (\Z^d)^{n-1}} \sigma(x_0) \mathbf{K}(x_0-x_1) P^\perp \sigma(x_1) \ldots \mathbf{K}(x_{n-1}-x_n) P^\perp \sigma(x_n),
\end{aligned}
\eeq
where we denote by $\ul{x}$ the vector $(x_1,x_2,\ldots,x_{n-1})$. We interpret this expression as a sum over \emph{paths} $\ul{x}$ in $\Z^d$ connecting $x_0$ to $x_n$. Along these paths, one evaluates the random variables $\sigma(x_k,\cdot)$. In between two such evaluations, one uses the (matrix-valued) ``propagator'' $\mathbf K(x_k-x_{k+1})$ to travel from site to site.

The first idea is then to treat \eqref{eq:tobound} as a composition of deterministic operators. For each $1\leq j\leq n$, define $K^j(x,y):=K_j(x-y)b_j(y)$ for a singular integral operator $K_j$ and a bounded function $b_j$ on $\Z^d$. For a given subset $S\subset (\Z^d)^{n-1}$, define the deterministic operator
\begin{equation}\label{eq:tns}
T^n_S(x_0,x_n)= \sum_{\ul{x}\in S} K^1(x_0,x_1)K^2(x_1,x_2)\ldots K^n(x_{n-1},x_n).
\end{equation}
Note that $T^n_{(\Z^d)^{n-1}}(x_0,x_n)$ is the kernel of the operator $K^1 K^2 \cdots K^n$. After replacing $P^\perp$ by $I - P$ in \eqref{eq:tobound}, we may control \eqref{eq:tobound} by a bound for $T^n_{(\Z^d)^{n-1}}(x_0,x_n)$. (In the application to \eqref{eq:tobound}, $K_j$ is a matrix element of $\mathbf{K}=\nabla (\De)^{-1}\nabla^*$ and $b$ is a realization of the random variable $\sigma$.) It follows from \cite[Lemma 1]{B} (see also Lemma \ref{lm:bourgain}) that
\beq\label{eq:lm1}
 |T^n_{(\Z^d)^{n-1}}(x_0,x_n)| \leq \epsilon \l(\frac{C}{\eps}\r)^n \br{x_0-x_n}^{-d+\epsilon},
\eeq
with the decay rate of $-d+\epsilon$. 

This shows that the deterministic estimate \eqref{eq:lm1} is not enough -- the randomness must be utilized. In order to discuss the role of randomness and the projection operators $P^\perp$ in the sum \eqref{eq:tobound}, we define 
\begin{defn}[Reducible paths]\label{defn:reducible}
Let $n\geq 2$ and fix $x_0,x_n\in \Z^d$  with $x_0\neq x_n$. We say that $\ux=(x_1,x_2,\ldots,x_{n-1})\in(\Z^d)^{n-1}$ is a reducible path (from $x_0$ to $x_n$) if there exists $0\leq j<n$ such that \[ \{ x_0,\ldots,x_j \}\cap \{ x_{j+1}, \ldots, x_n \} = \emptyset. \]
Otherwise we say that $\ux$ is an irreducible path. 
\end{defn}

The importance of this notion stems from the fact that we may discard \emph{any} portion of reducible paths $\ul{x}$ from the summation in \eqref{eq:tobound}. Indeed, if $\ul{x}=(x_1,\ldots,x_{n-1})$ is a reducible path (from $x_0$ to $x_n$), then
\begin{equation}\label{eqn:cancellation}
P\sigma(x_0,\cdot)P^\perp\sigma(x_1,\cdot)P^\perp\sigma(x_2,\cdot)\ldots P^\perp\sigma(x_n,\cdot)=0.
\end{equation}
This fundamental vanishing property follows from the assumption that the random variables are independent. In principle, this is a promising observation because it allows one to discard terms from the summation in \eqref{eq:tobound}. In effect, the sum over $(\Z^d)^{n-1}$ in \eqref{eq:tobound} can be replaced by one over appropriate subsets $S\subset (\Z^d)^{n-1}$. The discarding of reducible paths is the only way in which the randomness is utilized. Afterwards, the remaining task is to bound the deterministic quantity $T_S^n(x_0,x_n)$ for the selected $S\subset (\Z^d)^{n-1}$.

This touches upon a central, but subtle, issue: Precisely which reducible paths should be discarded from the sum \eqref{eq:tobound} (in other words, which $T_S^n(x_0,x_n)$ one should aim to bound) is not at all clear a priori. To avoid confusion, we emphasize that once one has reduced matters to the deterministic quantity $T^n_S(x_0,x_n)$, one may no longer drop reducible paths. Moreover, $T^n_S(x_0,x_n)$ does not depend on $S$ in a monotone way. In fact, the summation involves significant cancellations due to the presence of the singular integral operators and one should avoid taking absolute values inside the sum if possible. (From this perspective, Bourgain's deterministic bound \eqref{eq:lm1} is already non-trivial; see Remark \ref{rmk:fubini}.) To summarize, the \emph{main technical difficulty} is the delicate matter of bounding the oscillatory object $T^n_S(x_0,x_n)$ for a subset $S$ obtained by discarding an appropriate subset of reducible paths. 

Bourgain uses the vanishing property in the following way. By a dyadic decomposition, one may focus on the sum in \eqref{eq:tobound} over paths $\ux$ such that, for some fixed $0\leq j_0 < n$, the length of their longest segment $\max_{0\leq j<n}|x_j - x_{j+1}|$ is equal to $|x_{j_0}-x_{j_0+1}|$ and is comparable to $R$ for some large $R>0$. Let $S_{j_0}$ be the collection of such paths. Next, using the identity \eqref{eqn:cancellation}, he discards from $S_{j_0}$ exactly those reducible paths where the sub-paths $(x_0,\ldots,x_{j_0})$ and $(x_{j_0+1},\ldots,x_{n})$ are not connected. In other words, Bourgain only keeps paths in the set
\beq\label{eq:Sintro}
\tilde{S}_{j_0}:=\bigcup_{j_1\leq j_0<j_2}S_{j_1,j_2}, \;\text{ where} \;\; S_{j_1,j_2}:=\{ \ux \in  S_{j_0} : x_{j_1} = x_{j_2} \}.
\eeq

Thanks to \eqref{eq:lm1} and the structure of $S_{j_1,j_2}$, it is possible to control the sum \eqref{eq:tobound} restricted to the subset $S_{j_1,j_2}$ by 
\begin{equation}\label{eq:lm2}
|T^n_{S_{j_1,j_2}}(x_0,x_n)|\leq \epsilon^2 \l(\frac{C}{\eps}\r)^n R^{-d} \br{x_0-x_n}^{-d+\epsilon}. 
\end{equation}
Since $R\geq C |x_0-x_n|/n$, this already shows the decay rate of $-2d+\epsilon$ obtained in \cite{B}. However, a key point is that the union in \eqref{eq:Sintro} is \emph{not disjoint} and therefore a bound on the individual $T^n_{S_{j_1,j_2}}$ does not directly imply a bound on $T^n_{\tilde{S}_{j_0}}$. (We emphasize that this issue is a consequence of the oscillatory nature of the problem. If the definition \eqref{eq:tns} of $T^n_S$ would only involve positive terms, this step would follow by a simple union bound.) This a priori serious technical problem is solved in a highly original way in \cite{B} by introducing Steinhaus systems and appealing to the Markov brothers' inequality for polynomials. We call this as ``Bourgain's disjointification trick'' and abstract it to Lemma \ref{lm:bourgain2}. Altogether, Bourgain's argument gives the decay rate $-2d+\eps$.

 
 
Our improved decay rate starts with a simple observation: for each path $\ux=(x_1,\cdots,x_{n-1}) \in S_{j_1,j_2}$, we have 
\[ |x_0-x_n| \leq |x_0- x_{j_1}| + |x_{j_2}-x_n|\]
by the triangle inequality and $x_{j_1} = x_{j_2}$. This observation implies that there exists another ``long" segment  among the sub-paths $(x_0,\ldots,x_{j_1})$ or $(x_{j_2},\ldots,x_{n})$. Exploiting this additional information, we further decompose the set $S_{j_1,j_2}$ and discard certain reducible paths using the identity \eqref{eqn:cancellation} once more. These steps amount to specifying even smaller subsets $S\subset (\Z^d)^{n-1}$ for which \eqref{eq:tns} is to be controlled. We show an improved bound for $T^n_S(x_0,x_n)$ using additional structures in $S$ and then obtain \eqref{eq:main}, i.e., the decay rate $-3d+\eps$.


\subsection{A related question and a partial result}\label{sec:question}
As described above, any successful argument has to negotiate how many reducible paths to discard from the summation -- because afterwards one needs to control $T^n_S(x_0,x_n)$ on the resulting set $S$ of paths. Bourgain implements the cancellation \eqref{eqn:cancellation} once in his argument and we implement it twice to prove Theorem \ref{thm:main}. 

Now, what happens if we discard \emph{all} the reducible paths from the outset? Our result in this direction, Theorem \ref{thm:nlogn}, succeeds almost in yielding another proof of Theorem \ref{thm:main} (up to a logarithm). 

Let $n\geq 2$ and fix $x_0,x_n\in\Z^d$ with $x_0\neq x_n$. We denote by $U=U_{x_0,x_n}$ the set of all irreducible paths from $x_0$ to $x_n$, i.e.,
\[ U := \{ \ux \in(\Z^d)^{n-1}: \{ x_0, \ldots, x_j \} \cap \{ x_{j+1}, \ldots, x_n \} \neq \emptyset, \text{ for any } 0\leq j < n \}.\] 
Note that $U_{x_0,x_2} = \emptyset$ when $n=2$, and $U_{x_0,x_3} = \{ (x_3,x_0) \}$ when $n=3$. The set $U_{x_0,x_n}$ becomes more complicated when $n\geq 4$.  

Note that by the vanishing property, \eqref{eqn:cancellation}, we have
\begin{equation}\label{eqn:irreduible}
\begin{aligned}
&P \sigma (\mathbf{K} P^\perp \sigma)^n(x_0,x_n)\\
=&P\sum\limits_{\ul{x}\in U} \sigma(x_0) \mathbf{K}(x_0-x_1) P^\perp \sigma(x_1) \ldots \mathbf{K}(x_{n-1}-x_n) P^\perp \sigma(x_n).
\end{aligned}
\end{equation}
Each matrix element of the right-hand side of \eqref{eqn:irreduible} can be controlled by a sum of deterministic terms $T^n_U(x_0,x_n)$ defined in \eqref{eq:tns}; see Section \ref{sec:key}. 

Our second main result provides a non-trivial estimate for $T^n_U(x_0,x_n)$.
\begin{thm} \label{thm:nlogn} Let $n\geq 3$ and $x_0,x_n\in \Z^d$, $x_0\neq x_n$. Then there is an absolute constant $C>1$ such that
\[ |T^n_U(x_0,x_n)| \leq C^{n\log n} \epsilon^{3-n} \br{x_0 - x_n}^{-3d+\epsilon} \]
for all sufficiently small $\epsilon>0$.
\end{thm}

One may compare Theorem \ref{thm:nlogn} with the trivial estimate
\[ |T^n_U(x_0,x_n)| \leq C^{n\log n}\epsilon^{1-n} \br{x_0 - x_n}^{-d+\epsilon}\]
which holds for any $U\subset (\Z^d)^{n-1}$, see \eqref{eq:log}.

It would be very interesting to know whether it is possible to improve the constant $C^{n \log n}$ to $C^n$ in Theorem \ref{thm:nlogn}, which would then imply Theorem \ref{thm:main} arguing as in Section \ref{sec:key}. In fact, we show that we may write $U= \bigcup_{\alpha \in \mathcal{A}} U_\alpha$ for some index set $\mathcal{A}$ with $\# \mathcal{A} \leq 2^n$ such that 
\[ |T^n_{U_\alpha}(x_0,x_n)| \leq \epsilon^3 (C/\epsilon)^n \br{x_0-x_n}^{-3d+\epsilon}.\]
Since the sets $U_\alpha$ are not disjoint, this does not immediately yield a bound on $T^n_U(x_0,x_n)$. Nonetheless, we can perform an appropriate ``disjointification" to write $U= \bigsqcup_{\alpha \in \mathcal{A}} U_\alpha'$. (Here $\sqcup$ denotes disjoint union.) Unfortunately, the most efficient way to implement this disjointification that we have found still produces the $C^{n\log n}$ bound in Theorem \ref{thm:nlogn}.

\paragraph{Acknowledgments}
The authors would like to thank Wilhelm Schlag and Tom Spencer for helpful discussions.
This material is based upon work supported by the National Science Foundation under Grant No. DMS - 1638352.

\section{Preliminaries}\label{sec:pre}
\subsection{Invertibility of $L$ on some function spaces} \label{sec:spaces}
In this subsection, we discuss the invertibility of the operators $L_\om$ and $L$ on appropriate domains of definition; see Proposition \ref{prop:inverse} below. To this end, we introduce function spaces which play the role of Sobolev spaces in the discrete setting. This section furnishes the formal operator-theoretic foundation for the study of various objects in this paper and can be skipped upon a first reading.

We start the discussion with the identity $\widehat{-\Delta f} (\theta) = \sum_{j=1}^d 2(1-\cos \theta_j)  \hat{f}(\theta)$ for $\theta \in \T^d = [ -\pi, \pi]^d$, where the symbol $\sum_{j=1}^d 2(1-\cos \theta_j)$ of $-\Delta$ is comparable to $|\theta|^2$. For $f\in \ell^2(\Z^d)$ and $s>-d/2$, define the Riesz potential $\Lambda^{s} = (-\Delta)^{s/2}$ by
\[ \widehat{\Lambda^{s} f}(\theta) = \bigg(\sum_{j=1}^d 2(1-\cos \theta_j )\bigg)^{s/2} \hat{f}(\theta). \] 
We shall work with $\Lambda^1$ and $\Lambda^{-1}$ for $d\geq 3$.

We note that $\Lambda^{-1}$ is a bounded map from $\ell^2(\Z^d)$ to $\ell^{q_d}(\Z^d)$ for $q_d^{-1} := 2^{-1}-d^{-1}$. This is a consequence of a discrete version of the Hardy-Littlewood-Sobolev inequality:
\begin{equation}\label{eqn:HLS}
 \norm{\Lambda^{-1} f}_{\ell^q(\Z^d)} \leq C_{p,d} \norm{f}_{\ell^p(\Z^d)} 
\end{equation}
for $1<p\leq 2$ and $q^{-1}= p^{-1}-d^{-1}$. The estimate \eqref{eqn:HLS} follows from an estimate for a discrete analogue of fractional integrals on $\Z^d$ (see, e.g., \cite[Proposition (a)]{SW}) and the fact that $\Lambda^{-1}f = K*f$ for a convolution kernel $K\in \ell^2(\Z^d)$ satisfying the bound $K(x)=O((1+|x|)^{-d+1})$. See the proof of Corollary \ref{cor:GF} in Appendix for a related computation. 

We specify the domain of the map $L_\om$. Let $H^1(\Z^d)$ be the image of $\ell^2(\Z^d)$ by the injection $\Lambda^{-1}:\ell^2(\Z^d) \to \ell^{q_d}(\Z^d)$. Namely, 
 \[ H^1(\Z^d):=\Lambda^{-1}[\ell^2(\Z^d)]=\{ \Lambda^{-1} f : f\in \ell^2(\Z^d)\}.\] 
We equip $H^1(\Z^d)$ with the norm $\norm{\Lambda^{-1}f }_{H^1(\Z^d)} := \norm{f}_{\ell^2(\Z^d)}$. In fact, $H^1(\Z^d)$ is a Hilbert space equipped with the inner product $\inn{\Lambda^{-1}f_1, \Lambda^{-1}f_2 }_{H^1(\Z^d)} :=  \inn{f_1,f_2 }_{\ell^2(\Z^d)}$. We \emph{define} $\Lambda: H^1(\Z^d) \to \ell^2(\Z^d)$ to be the inverse of the map $\Lambda^{-1}$. By definition, the maps $\Lambda^{-1} : \ell^2(\Z^d) \to H^1(\Z^d)$ and $\Lambda: H^1(\Z^d) \to \ell^2(\Z^d)$ are isometries. 

The range of $L_\om$ can be identified with $H^{-1}(\Z^d)$ defined by
\[ H^{-1}(\Z^d) := \{ f\in \ell^2(\Z^d): \Lambda^{-1} f \in \ell^2(\Z^d) \}. \]
Note that $H^{-1}(\Z^d)$ is a Hilbert space equipped with the inner product $\inn{ f,g}_{H^{-1}(\Z^d)} := \inn{\Lambda^{-1} f,\Lambda^{-1} g}_{\ell^2(\Z^d)}$. The map $\Lambda^{-1}:H^{-1}(\Z^d) \to \ell^2(\Z^d)$ and its inverse $\Lambda^{1}:\ell^2(\Z^d) \to H^{-1}(\Z^d)$ are isometries. By \eqref{eqn:HLS}, we have $\ell^{p_d}(\Z^d) \subset H^{-1}(\Z^d)$ for $p_d^{-1}:=2^{-1}+d^{-1}$.

\begin{prop} \label{prop:inverse} Let $0<\delta<1$. The operator $L_\om: H^{1}(\Z^d) \to H^{-1}(\Z^d)$ is a bounded operator with the bounded inverse $L_\om^{-1}$. The operator norms for $L_\om$ and $L_\om^{-1}$ are bounded uniformly in $\omega$.
\end{prop}
\begin{proof} The proof is standard. First, observe that $\nabla_j: H^1(\Z^d) \to \ell^2(\Z^d)$ and $\nabla_j^* :\ell^2(\Z^d)\to H^{-1}(\Z^d)$ are bounded maps by Plancherel's theorem. Thus, $L_\om$ is a bounded map from $H^1(\Z^d)$ to $H^{-1}(\Z^d)$ uniformly in $\om\in \Omega$. 

Next, we observe that $-\Delta = \Lambda^1 \Lambda$ on $H^1(\Z^d)$. This is because \[ -\Delta (\Lambda^{-1} f) = \Lambda^1 f = \Lambda^1 \Lambda (\Lambda^{-1} f)\] 
for any $f\in \ell^2(\Z^d)$. This allows us to factorize $L_\om$ as
\[ L_\om = \Lambda^1 ( I + \delta M_\om ) \Lambda, \]
where $M_\om:= \Lambda^{-1} \nabla^*\sigma(\cdot, \om) \mathbf{I}_d \nabla \Lambda^{-1}$. The operator $M_\om$ is bounded on $\ell^2(\Z^d)$ with the operator norm bounded by $1$. Therefore, when $0< \delta<1$, the inverse of $I+\delta M_\om$ exists and is bounded on $\ell^2(\Z^d)$. This implies that $L_\om$ is invertible and $L_\om^{-1}$ is bounded (uniformly in $\omega$).
\end{proof}

This prototypical result also applies in a slightly different context which we will occasionally consider and which is therefore made precise next.

We may view $L = \nabla^* \mathbf{A} \nabla$ as a map acting on functions on the product space $\Z^d \times \Omega$ via $Lu(x,\om) = L_\om u_\om (x)$. One can show a completely analogous proposition, where the relevant function spaces are replaced by the following ones. We first let $L^2(\Z^d\times \Omega)$ be the Hilbert space equipped with the inner product induced from $\ell^2(\Z^d)$ and $L^2(\Omega)$. By letting $\Lambda^{-1}$ act on the lattice variable, we may regard it as a bounded injection from $L^2(\Z^d\times \Omega)$ to the mixed norm space $L^2(\Omega, \ell^{q_d}(\Z^d))$. Define the Hilbert spaces $H^1(\Z^d \times \Omega)= \Lambda^{-1}[ L^2(\Z^d\times \Omega) ] $ and 
\[ H^{-1}(\Z^d\times \Omega) = \{ f\in L^2(\Z^d\times \Omega): \Lambda^{-1} f \in L^2(\Z^d \times \Omega) \}. \]
Arguing as in the proof of Proposition \ref{prop:inverse}, one verifies that 
$$
L:H^{1}(\Z^d\times \Omega) \to H^{-1}(\Z^d\times \Omega)
$$
 is a bounded operator with the bounded inverse $L^{-1}$.



\subsection{Derivation of the perturbation series via Feshbach-Schur}\label{sec:derive}
Let $\hi_1 = H^1(\Z^d \times \Omega)$ and $\hi_2 = H^{-1}(\Z^d \times \Omega)$. In Section \ref{sec:spaces}, we have seen that $L: \hi_1 \to \hi_2$ and its inverse $L^{-1}:\hi_2 \to \hi_1$ are bounded operators. Let $P=\mathbb E[\cdot]$ be the projection operator acting on $L^2(\Z^d\times \Omega)$ and $P^\perp := I-P$. We may identify the operator $\curly {L}=(\mathbb{E} [L_\om^{-1}])^{-1}$ with the inverse of $PL^{-1}P : P\hi_2 \to P\hi_1$. To compute the inverse of $PL^{-1}P$, following \cite{Sigal}, one decomposes the operator $L$ into blocks 
\[
L = \ttmatrix{PLP}{P LP^\perp}{P^\perp LP}{P^\perp LP^\perp}.
\]
 
When $P^\perp L P^\perp : P^\perp \hi_1 \to P^\perp \hi_2$ is invertible, the inverse of $PL^{-1}P$ exists and is given by the Feshbach-Schur map (also called Schur complement formula)
\begin{equation}\label{eqn:Fesh}
\curly{L} = (PL^{-1}P)^{-1} = PLP - PLP^\perp (P^\perp L P^\perp)^{-1} P^\perp L P.
\end{equation}
We now check the invertibility of $P^\perp L P^\perp$. We write
\[ P^\perp L P^{\perp}  = \Lambda^{1}(I+\delta M^\perp) \Lambda P^\perp, \]
where $M^\perp = \Lambda^{-1} \nabla^* P^\perp \sigma \mathbf{I}_d \nabla \Lambda^{-1}$. The $L^2(\Z^d\times \Omega)$ operator norm of $M^\perp$ is bounded by $1$. Therefore $(I+\delta M^\perp)$ is invertible when $0<\delta<1$ and the inverse can be written as a Neumann series. One verifies that the inverse of $P^\perp L P^\perp$ is given by 
\[ (P^\perp L P^{\perp})^{-1}  = \Lambda^{-1}(I+\delta M^\perp)^{-1} \Lambda^{-1} P^\perp. \]
The upshot of these considerations, which we do not repeat here, is the expression \eqref{eq:series} for the operator $\mathbf{K}^\de$. See \cite{B} for details.

Finally, we also have that $\curly{L}:H^{1}(\Z^d) \to H^{-1}(\Z^d)$ is a bounded operator with bounded inverse, whenever $0<\de<1$. This follows from the expression \eqref{eqn:Fesh} and the boundedness of the operators $L$ and $L^{-1}$.


\subsection{Bourgain's lemmas} \label{sec:bourgain}

In this subsection, we state abstract versions of two main tools introduced in \cite{B}: a deterministic bound on composition of singular integral operators and Bourgain's disjointification trick.

Before we proceed, we briefly recall some well-known properties of singular integral operators to be used later.

By a \emph{singular integral operator} (of convolution type) acting on $\Z^d$, we mean, in this paper, a Fourier multiplier transformation $K$ of the form $\widehat{Ku} (\theta) = m(\theta) \hat{u}(\theta)$ associated with a multiplier $m$ on the $d$-torus $\T^d$ satisfying the bounds
\[ |\partial^\alpha m (\theta)| \leq C_\alpha |\theta|^{-|\alpha|} \]
for all multi-indices $|\alpha|\geq 0$. Here $\hat{u}(\theta)=\sum_{x\in \Z^d} u(x) e^{-i\theta\cdot x}$ denotes the Fourier transform of a function $u$ on $\Z^d$. The convolution kernel $K(x)$ of such $K$ satisfies the decay estimate $|K(x)| \leq C(1+|x|)^{-d}$ and the ``gradient" estimate $|\nabla K(x)| \leq C  (1+|x|)^{-(d+1)}$. The $\ell^2(\Z^d)$ boundedness of the convolution operator $K$ follows from the boundedness of $m$. It is also well-known, by the Calder\'{o}n-Zygmund theory, that $K$ is of weak-type $(1,1)$, hence bounded on $\ell^p(\Z^d)$ for all $1<p<\infty$ by interpolation and duality with the operator norm $O((p-1)^{-1})$ as $p\to 1$. See \cite{Stein} for a treatment of singular integrals in the continuous setting.

\subsubsection{Deterministic bounds}
We recall that we identify an operator $K$ with its kernel. We write $K^*$ for the adjoint of $K$ and so $K^*(x,y) = \overline{ K(y,x) }$. For an interval $I\subset [0,\infty)$, we set 
\beq\label{eq:KIdefn}
 K_I(x,y) := K(x,y) \chi_{I}(|x-y|).
 \eeq

\begin{lm} \label{lm:bourgain} Let $A>0$ and let $\eps>0$ be sufficiently small. There exists a constant $C=C_{d,A}$ such that the following holds. Let $\mathbf{I}_m = [0,2^m)$ and let $\{K^j\}_{1\leq j\leq n}$ be a collection of operators acting on functions on $\Z^d$ satisfying the assumptions
\begin{enumerate}
  \item[(i)] $ |K^j(x,y)|\leq A\br{x-y}^{-d}$. 
  \item[(ii)] $\sup_{m\geq 0}\l( \norm{K^j_{\mathbf{I}_m}}_{\ell^p \to \ell^p} + \norm{(K^j_{\mathbf{I}_m})^*}_{\ell^p \to \ell^p} \r)\leq A/(p-1)$ for all sufficiently small $p>1$.
 \end{enumerate}
Then the operator $
 T^n:=K^1 K^2 \ldots K^n
$
satisfies the pointwise bound 
\[
 |T^n(x_0,x_n)| \leq  \epsilon \l(\frac{C}{\eps}\r)^{n}\br{x_0-x_n}^{-d+\epsilon} 
\]
for all sufficiently small $\eps>0$.
\end{lm}

Lemma \ref{lm:bourgain} is an abstract version of Lemma 1 in \cite{B}, where the operators $K^j$ are of the form $K^j(x,y) = K(x-y) b(y)$ for a singular integral operator $K$ and a bounded function $b$, see Example \ref{ex:cz}. The main motivation to state Lemma \ref{lm:bourgain} in this generality is that it makes it easier to apply the result for variants of $K^j$; see Corollary \ref{cor:bourgain}. While the bound in Lemma 1 of \cite{B} features $(C/\epsilon)^{n}$, we note here that its proof in fact yields an additional factor of $\epsilon$. This gain allows for the improvements described in Remark \ref{rmk:main} (i). We present the proof of Lemma \ref{lm:bourgain} in the appendix for completeness.

We give the main example for operators $\{ K^j \}_{1\leq j\leq n}$ for Lemma \ref{lm:bourgain}.
\begin{exam}\label{ex:cz}
Consider the example from Lemma 1 in \cite{B}:
\[ K^j(x,y) = K(x-y) b(y)\] for a singular integral operator $K$ of convolution type acting on $\Z^d$ (e.g.,  $K=\nabla_j (-\Delta)^{-1} \nabla_k^*$) and a function $b \in L^\infty(\Z^d)$. To check Assumption (i) and (ii), it is enough to assume that $K^j(x,y)=K(x-y)$ since $b\in L^\infty(\Z^d)$. Assumption (i) is immediate from our assumption on $K$. To verify Assumption (ii), we compare $K_{\mathbf{I}_m}$ with its variant $\tilde{K}_{\mathbf{I}_m}$, where the sharp cutoff $\chi_{I_m}(|x-y|)$ is replaced by a smooth cutoff $\psi(2^{-m}|x-y|)$. Here, $\psi$ is a smooth even function supported on $[-2,2]$ and equal to $1$ on $[-1,1]$. Note that $\norm{K_{\mathbf{I}_m} - \tilde{K}_{\mathbf{I}_m}}_{\ell^1(\Z^d)} \leq C$ and thus the operator $K_{\mathbf{I}_m} - \tilde{K}_{\mathbf{I}_m}$ is uniformly bounded on $\ell^p(\Z^d)$ for $1\leq p\leq \infty$. Therefore, it suffices to check Assumption (ii) for $\tilde{K}_{\mathbf{I}_m}$, which is well-known by the Calder\'{o}n-Zygmund theory.
\end{exam}

In fact, Lemma \ref{lm:bourgain} has a slightly wider scope than its statement suggests. We state a specific version in the following corollary and use it later with $K^j$ as in Example \ref{ex:cz}.
\begin{cor} \label{cor:bourgain}
Let $A>0$ and let $\eps>0$ be sufficiently small. There exists a constant $C=C_{d,A}$ such that the following holds. Let $\{K^j\}_{1\leq j\leq n}$ be a collection of operators as in Lemma \ref{lm:bourgain} and let $\curly{I}=\{I_j \}_{1\leq j\leq n}$ be a collection of intervals $I_j \subset [0,\infty)$. Then 
\begin{equation}\label{eq:bgt}
 |K^1_{I_1} K^2_{I_2} \ldots K^n_{I_n}(x_0,x_n)| \leq  \epsilon \l(\frac{C}{\eps}\r)^{n}\br{x_0-x_n}^{-d+\epsilon} 
\end{equation}
for all sufficiently small $\epsilon>0$. Moreover, the bound \eqref{eq:bgt} is invariant under the change 
\begin{equation}\label{eq:bgt2}
K^j_{I_j}(x,y) \to e^j_{1}(x)K^j_{I_j}(x,y)e^j_{2}(y)
\end{equation}
for any $e^j_1,e^j_2 \in \ell^\infty(\Z^d)$ such that $\norm{e^j_{1}}_{\ell^\infty} \leq 1$ and $\norm{e^j_{2}}_{\ell^\infty} \leq 1$.
\end{cor}
\begin{proof}
We first note that $K^j_{I_j}$ satisfies Assumptions (i) and (ii) with another constant $A'>0$ \emph{independent} of $I_j$. This can be shown by choosing an interval of the forms $\mathbf{I}_m$ or $\mathbf{I}_{m} \setminus \mathbf{I}_{m_0}$ for certain $m\geq m_0\geq 1$ that best approximates $I_j$ and arguing as in Example \ref{ex:cz} using Assumption (ii) for $K^j_{\mathbf{I}_m}$. This verifies \eqref{eq:bgt} by Lemma \ref{lm:bourgain}. Moreover, this bound is invariant under \eqref{eq:bgt2} since the bounds in Assumption (i) and (ii) are preserved under the replacement \eqref{eq:bgt2}.
\end{proof}

\be{rmk}\label{rmk:fubini}
The constant $(C/\epsilon)^n$ obtained in \cite[Lemma 1]{B} (and also stated in Corollary \ref{cor:bourgain}) is rather non-trivial. To compare, we note that there is a much weaker bound that only requires (a weaker version of) the size assumption $|K^j(x,y)| \leq A \br{x-y}^{-d}$: $|T^n(x_0,x_n)|$ is bounded by
\begin{equation}\label{eq:log}
\begin{split}
 \sum_{\ul{x}\in (\Z^d)^{n-1}} |K^1(x_0,x_1)K^2(x_1,x_2)\ldots K^n(x_{n-1} ,x_n)| \leq  \epsilon\left(\frac{nC}{\epsilon}\right)^n \br{x_0-x_n}^{-d+\epsilon}
\end{split}
\end{equation}
for all $0<\epsilon<d/4$. Note that Lemma \ref{lm:bourgain} saves a factor of $n^n$ compared to \eqref{eq:log}. (We still use \eqref{eq:log} -- to justify the use of Fubini's theorem on various occasions throughout this note.)

To see that \eqref{eq:log} holds, one may use the bound $|K^j(x,y)| \leq A \br{x-y}^{-d+(\epsilon/n)}$ and then apply the following elementary inequality $(n-1)$-times. 
\e{rmk}
\begin{lm}\label{lm:sum}
Assume that $\alpha,\beta$ are positive numbers such that $3d/4 \leq \alpha,\beta \leq d-\epsilon$ for some $\epsilon>0$. Then there exists a constant $C$ such that for any $a,b\in \Z^d$, we have
\[ \sum_{x\in\Z^d} \br{a-x}^{-\alpha}\br{x-b}^{-\beta} \leq C \epsilon^{-1} \br{a-b}^{-(\alpha+\beta-d)}. \]
\end{lm}
One way to verify Lemma \ref{lm:sum} is to make a dyadic decomposition $\Z^d = \bigcup_{m \geq 0} \{ x\in \Z^d: \max(|a-x|,|x-b|) \sim 2^m\} $ as in the proof of Lemma \ref{lm:bourgain} (see the appendix). We leave the detail to the interested reader.

\subsubsection{Bourgain's disjointification trick}
One of the main technical challenges that \cite{B} overcomes is bounding $T^n_{S}$ for rather small $S\subset (\Z^d)^{n-1}$. In the proof, after this is achieved for certain sets $S$, it remains to add appropriate disjointness conditions, resulting in even smaller sets $S'\subset S$. Bourgain's trick then gives a way to bound $|T^n_{S'}(x_0,x_n)|$ in terms of a bound for $|T^n_{S}(x_0,x_n)|$, up to a factor of $C^n$.

We slightly generalize Bourgain's trick in the following lemma. 

\begin{lm}\label{lm:bourgain2}
Let $\{K^j\}_{1\leq j\leq n}$ be a collection of operators as in Lemma \ref{lm:bourgain} and define $T^n_S$ as in \eqref{eq:tns} for $S \subset (\Z^d)^{n-1}$. For given subsets $\{E_l,F_l \}_{1\leq l\leq m}$ of $\{0,1,2,\ldots,n\}$, define 
\[ S' := S \cap \bigcap_{1\leq l \leq m} \{ \ux\in(\Z^d)^{n-1}: \{ x_u : u\in E_l \} \cap \{ x_v : v\in F_l \} = \emptyset \}. \]

Assume that $S$ is a finite set and we have
\[
|T^n_S(x_0,x_n)| \leq M(x_0,x_n)
\]
for some function $M$ and that the estimate remains invariant under the change 
\begin{equation}\label{eqn:change}
K^j(x,y) \to  e^j_{1}(x)K^j(x,y)e^j_{2}(y).
\end{equation}
Then we have
\begin{equation}\label{eqn:bg2}
 |T^n_{S'}(x_0,x_n)| \leq 2^{\sum_{1\leq l\leq m} |E_l|+|F_l|} M(x_0,x_n).
\end{equation}

When $S$ is not finite, consider the truncation $S\cap X_k$, where 
\[ X_k := \{ \ul{x}\in (\Z^d)^{n-1}: \max_{0\leq j<n }|x_j-x_{j+1}| < 2^k \}. \]  
Then \eqref{eqn:bg2} holds if
\[ |T^n_{S\cap X_k} (x_0,x_n)| \leq M(x_0,x_n)\]
holds for all large $k\geq 1$ and the estimate remains invariant under the change \eqref{eqn:change}.
\end{lm}

The proof, which we relegate to the appendix, follows \cite{B} and uses Steinhaus systems and the Markov brothers' inequality for polynomials.

\section{Proof of the main result}\label{sec:mainresult}
\subsection{The key estimate}\label{sec:key}
We express the right-hand side of \eqref{eq:series} in terms of paths in $\Z^d$. We let $K_1,\ldots,K_n$ be singular integral operators of the form $\nabla_{i} (-\Delta)^{-1} \nabla_{i'}^*$ for some $1\leq i,i'\leq d$. This specific choice, however, is not important for the argument. For every $n\geq 1$ and every subset $X\subset (\Z^d)^{n-1}$, define a function $f_X:\Z^d\times\Z^d\to \R$ by
\beq\label{eq:fdefn}
\begin{aligned}
&f_X(x_0,x_n)\\
&:=P\sigma(x_0) \sum\limits_{\ul{x}\in X} K_1(x_0-x_1) P^\perp \sigma(x_1) \cdots K_n(x_{n-1}-x_n)P^\perp \sigma(x_n).
\end{aligned}
\eeq
Here and in the following, we denote $\ul{x}=(x_1,\ldots,x_{n-1})$ and we suppress the randomness from the notation. 

It should be noted that $T^n_X$, studied in Section \ref{sec:pre}, is a deterministic version of $f_X$. In the following, we indicate how to obtain a bound for $f_X(x_0,x_n)$ from a bound for $T^n_X(x_0,x_n)$.
\be{lm}\label{lm:transfer} Let $X\subset (\Z^d)^{n-1}$ and $b_1,\ldots,b_n:\Z^d\to\C$ be functions such that $\norm{b_j}_\infty \leq 1$. Define operators $K^j$ via their kernels $K^j(x,y):=K_j(x-y)b_j(y)$. Assume that 
\[ T^n_{X}(x_0,x_n) = \sum\limits_{\ul{x}\in X} K^1(x_0,x_1)K^2(x_1,x_2)\ldots K^n(x_{n-1},x_n)
\]
satisfies the bound
\[ |T^n_X(x_0,x_n)| \leq M(x_0,x_n) \]
for some function $M$ independent of the choice of $\{b_j\}_{1\leq j\leq n}$. Then
\[ |f_X(x_0,x_n)| \leq 2^{n} M(x_0,x_n). \]
\e{lm}
\be{proof}
We replace each $P^\perp$ with $I-P=I-\mathbb E$ in \eqref{eq:fdefn}. This allows us to write $f_X(x_0,x_n)$ as a sum of $2^n$ terms of the form \eqref{eq:fdefn}, where each $P^\perp$ is replaced by either $I$ or $-\mathbb E$. For each of these terms, we use Fubini's theorem to move all the integrations corresponding to $\mathbb E$ outside of the sum $\sum_{\ux\in X}$. The proof is completed by bounding the sum over $\ux\in X$ using the assumption on $T^n_X$ with $b_j(x)=\sigma(x,\om_j)$ for some $\om_j \in \Omega$. (The assumption $\norm{b_j}_\infty \leq 1$ is guaranteed because $|\sigma(x,\om)|\leq 1$ holds for almost every $\om$ and all the $\om_j$ appear under an integral.)
\e{proof}

Lemma \ref{lm:bourgain}, Example \ref{ex:cz}, and Lemma \ref{lm:transfer} imply that 
\begin{equation}\label{eq:basic}
|f_{(\Z^d)^{n-1}}(x_0,x_n)| \leq  \epsilon \l(\frac{C}{\eps}\r)^{n}\br{x_0-x_n}^{-d+\epsilon}.
\end{equation}
Our main result is a consequence of the following improved estimate.

\be{prop}[Key estimate]\label{prop:key}
There exists a constant $C>0$ such that the following holds. For every $1\leq j\leq n$, let $K_j=\nabla_{i_j} (-\Delta)^{-1} \nabla_{i_j'}^*$ for some $1\leq i_j,i_j'\leq d$. Then, for every $\eps>0$, $n\geq 3$ and distinct $x_0,x_n\in\Z^d$, we have
\beq\label{eq:propkey}
|f_{(\Z^d)^{n-1}}(x_0,x_n)|\leq \eps^3\l(\frac{C}{\eps}\r)^n \jap{x_0-x_n}^{-3d+\eps}.
\eeq
\e{prop}
\be{proof}[Proof of Theorem \ref{thm:main} assuming Proposition \ref{prop:key}]
We recall what \eqref{eq:series} says element-wise, i.e.,
$$
K^\de_{i,i'}=\delta \sum_{n=1}^\it (-\de)^n P \sigma (\mathbf{K} P^\perp \sigma)_{i,i'}^n
$$
for every $1\leq i,i'\leq d$. When we write out the matrix product $(\mathbf{K} P^\perp \sigma)_{i,i'}^n$, we obtain a sum of terms defined as in \eqref{eq:fdefn}, with $K_j=\nabla_{i_j} (-\Delta)^{-1} \nabla_{i_j'}^*$ for some $1\leq i_j,i_j'\leq d$. For each choice of the ``outside indices'' $i,i'$, there are $d^{n-1}$ choices of $K_j$ (for every $n$), and so
\beq\label{eq:series2}
\begin{aligned}
|K^\de_{i,i'}(x,y)|
=&\delta \l|\sum_{n=1}^\it (-\de)^n P \sigma (\mathbf{K} P^\perp \sigma)_{i,i'}^n(x,y)\r|\\
\leq& \delta \sum_{n=1}^\it  \de^n d^{n-1} \max_{K_1,\ldots,K_n}\l| f_{(\Z^d)^{n-1}}(x,y)\r|.
\end{aligned}
\eeq
The maximum is taken over operators $K_1,\ldots,K_n$ of the form $\nabla_{i_j} (-\Delta)^{-1} \nabla_{i_j'}^*$ for some $1\leq i_j,i_j'\leq d$. Using \eqref{eq:basic}, we see that 
$$\begin{aligned}
|K^\de_{i,i'}(x,y)|
\leq& \delta \eps \sum_{n=1}^\it  \l(\frac{C\de d}{\eps}\r)^n \jap{x-y}^{-d+\eps}
=C \delta^2 \jap{x-y}^{-d+\eps}
\end{aligned}
$$
whenever $0<\delta < c \epsilon$ with $c=(Cd)^{-1}$. This estimate, in particular, verifies the case $x=y$ in Theorem \ref{thm:main}.

When $x\neq y$, we use Proposition \ref{prop:key} (instead of \eqref{eq:basic} as above) together with the following observation: $f_{(\Z^d)^{n-1}}(x,y)=0$ for $n=1,2$ because 
$$
P\sig(x,\cdot)P^\perp \sig(y,\cdot)=0,\qquad P\sig(x,\cdot)P^\perp \sig(x_1,\cdot)P^\perp \sig(y,\cdot)=0,
$$
where the second equality holds for all $x_1\in\Z^d$. (Equivalently, when $n=2$, all paths connecting $x\neq y$ are reducible in the sense of Definition \ref{defn:reducible}.)
 
The fact that $\mathbf{K}^\de$ is a convolution operator, i.e., that $K^\de_{i,i'}(x_0,x_n)=K^\de_{i,i'}(x_0-x_n,0)$ follows from \eqref{eq:fdefn}: We change the summation variables $x_k\to x_k-x_n$ and recall that the random variables $\{\sigma(x,\om)\}_{x\in (\Z^d)^{n-1}}$ are identically distributed. 

Finally, one can derive the regularity properties of the Fourier transform $\ft{K^{\de}_{i,j}}$ from the decay estimate \eqref{eq:main} by standard arguments (mainly integration by parts). This finishes the proof of Theorem \ref{thm:main}.
\e{proof}

In the remainder of this section, we prove Proposition \ref{prop:key} by successively reducing it to simpler statements.

As we mentioned before, the basic observation behind our proof is that a second ``long'' segment exists in every path analyzed in \cite{B} by the triangle inequality. Our contribution starts at the conclusion of Bourgain's argument. Therefore we repeat Bourgain's argument here, and we include some additional details, before we show how to go a step further.\\

\textbf{Preparations.} From now on, we fix $\eps>0$, $n\geq 3$ and $x_0,x_n\in \Z^d$ with $x_0\neq x_n$. For every $1\leq j\leq n$, we let $K_j=\nabla_{i_j} (-\Delta)^{-1} \nabla_{i_j'}^*$ for some $1\leq i_j,i_j'\leq d$ whose values do not matter in what follows. 

\subsection{Dyadic decomposition by longest segment}
We begin by making precise the dyadic decomposition used by Bourgain to prove Lemma 1. It decomposes the paths $\ul{x}=(x_1,\ldots,x_{n-1})$ according to the dyadic scale of their ``longest'' segment $|x_{j}-x_{j+1}|$. 

We recall Definition \eqref{eq:fdefn}:
$$
\begin{aligned}
f_{(\Z^d)^{n-1}}(x_0,x_n)=P \sigma(x_0)\sum_{\ul{x}\in (\Z^d)^{n-1}}  K_1(x_0-x_1) P^\perp \sigma(x_1) K_2(x_1-x_2) \ldots \sigma(x_n),
\end{aligned}
$$
where we write $\ul{x}=(x_1,\ldots,x_{n-1})$. Using a dyadic decomposition according to the size of $\max_{0\leq j<n} |x_j-x_{j+1}|$, we may decompose the sum over paths as follows:
\beq\label{eq:sumdecomp}
f_{(\Z^d)^{n-1}}(x_0,x_n)
=\sum_{m=0}^\it \sum_{j_0=0}^{n-1}
f_{S_{j_0}^{m}}(x_0,x_n) .
\eeq
Here we introduced the family of disjoint sets
\beq\label{eq:Sjupperdefn}
\begin{aligned}
S_{j_0}^{m}
:=&\l\{\ul{x}\in (\Z^d)^{n-1}\,:\,
\max\limits_{0\leq  j<n}|x_j-x_{j+1}|< 2^{m+1},\quad 
|x_{j_0}-x_{j_0+1}|\geq 2^m\r.\\
&\;\qquad\qquad \l.\textnormal{ and } \max\limits_{0\leq j<j_0}|x_j-x_{j+1}|< 2^m\r\}.
\end{aligned}
\eeq
The last condition says that $j_0$ is minimal: it is the \emph{first} time that the path achieves the (dyadic scale of) the longest segment. The main objective is to estimate the sum over $S_{j_0}^{m}$ in \eqref{eq:sumdecomp}. In other words, we fix $m\geq0$, $0\leq j_0<n$ and focus on paths $\ul{x}$ such that
\beq\label{eq:Rdefn}
\max_{0\leq j< n} |x_j-x_{j+1}|\leq 2^{m+1},\qquad |x_{j_0}-x_{j_0+1}|\geq 2^m\quad \textnormal{ and $j_0$ is minimal.}
\eeq
(Compare eq.\ (4.2) in \cite{B}.)  

\subsection{Discarding reducible paths}
As we mentioned in the introduction, the main use of the probabilistic structure of the problem is that the contribution to \eqref{eq:fdefn} of every ``reducible'' path (i.e., a path that can be split into disjoint pieces) \emph{vanishes}, by Fubini's theorem. 


We define the family of disjoint sets
\beq\label{eq:Sdefn}
\tilde S_{j_0}^m:=\setof{\ul{x}\in S_{j_0}^{m}}{\{x_0,\ldots,x_{j_0}\}\cap \{x_{j_0+1},\ldots,x_{n}\}\neq \emptyset}.
\eeq

\be{lm}\label{lm:red1}
For all $m\geq0$ and $0\leq j_0< n$ we have
$$
f_{\tilde S_{j_0}^m}(x_0,x_n)
=f_{ S_{j_0}^m}(x_0,x_n).
$$
\e{lm}

\be{proof}
This holds because paths in $S_{j_0}^{m}\setminus \tilde S_{j_0}^{m}$ are reducible and thus do not contribute to the sum in \eqref{eq:fdefn}.
\e{proof}

Hence it suffices to prove the following estimate.

\be{prop}[Reduction 1]\label{prop:Skey}
Let $x_0\neq x_n$. For all $m\geq0$ and $0\leq j_0< n$, we have
\beq\label{eq:propSkey}
|f_{\tilde S_{j_0}^m}(x_0,x_n)|\leq \eps^3 \l(\frac{C}{\eps}\r)^{n}2^{-m(d-\eps)}\jap{x_0-x_n}^{-2d+\eps}.
\eeq
\e{prop}


Before we go on, we show that Proposition \ref{prop:Skey} implies the key estimate.

\be{proof}[Proof of Proposition \ref{prop:key} assuming Proposition \ref{prop:Skey}]\mbox{}\\
By \eqref{eq:sumdecomp}, Lemma \ref{lm:red1} and Proposition \ref{prop:Skey}, we have 
\beq\label{eq:combining}
\begin{aligned}
|f_{(\Z^d)^{n-1}}(x_0,x_n)|
\leq& \sum_{m=0}^\it \sum_{j_0=0}^{n-1} |f_{\tilde S_{j_0}^m}(x_0,x_n)|\\
\leq& n\eps^3 \l(\frac{C}{\eps}\r)^{n}\jap{x_0-x_n}^{-2d+\eps}  \sum_{\substack{m=0\\ \bigcup_{j_0}\tilde S_{j_0}^m\neq \emptyset} }^\it  2^{-m(d-\eps)}.
\end{aligned}
\eeq
From the Definition \eqref{eq:Sjupperdefn}, we see that $\bigcup_{j_0=0}^{n-1}\tilde S_{j_0}^m\neq \emptyset$ implies, by the triangle inequality, $|x_0-x_n|\leq n2^{m+1}$, or equivalently, 
$$
2^m\geq \frac{|x_0-x_n|}{2n}.
$$
We distinguish cases. If $|x_0-x_n|\leq 2n$, then we have $1\leq (1+2n)\jap{x_0-x_n}^{-1}$ and the claim follows easily from \eqref{eq:combining}. If $|x_0-x_n|>2n$, then, letting $q_n:=\ceil{\log_2(|x_0-x_n|/(2n))}$,
$$
\sum_{\substack{m=0\\ \bigcup_{j_0=0}\tilde S_{j_0}^m\neq \emptyset} }^\it  2^{-m(d-\eps)}
\leq \sum_{m=q_n}^{\it}2^{-m(d-\eps)}\leq C n^{d-\eps} \jap{x_0-x_n}^{-d+\eps}.
$$
Combining this with \eqref{eq:combining} yields the bound in Proposition \ref{prop:key}.
\e{proof}

We are left with the task of proving Proposition \ref{prop:Skey}. In the following, we always fix $m\geq0$ and $0\leq j_0< n$ and therefore we suppress them from the notation:
$$
S=\tilde S_{j_0}^m.
$$

\subsection{Decomposing the set $S$}

Observe that any path $\ul{x}\in S$ ($=\tilde S_{j_0}^m$ from \eqref{eq:Sdefn}) contains at least one coincidence point $x_{j_1}=x_{j_2}$, with $0\leq j_1\leq j_0$ and $j_0<j_2\leq n$. Following \cite{B}, we decompose the set $S$ according to where this coincidence occurs; see \eqref{eq:decomp1} below. (Afterwards, we show how an application of the triangle inequality implies that there exists a second ``long'' segment and so this procedure can be basically repeated; see \eqref{eq:decomp2} below.)

We define the sets
\beq\label{eq:Sjdefn}
\begin{aligned}
S_{j_1,j_2}:=&\setof{\ul{x}\in S}{x_{j_1}=x_{j_2}},\\
S_{j_0,j_1,j_2}':=&S_{j_1,j_2}\setminus\l(\bigcup_{\substack{j<j_1\\ j_0<j'\leq n}} S_{j,j'}\cup \bigcup_{j_0<j'<j_2} S_{j_1,j'}\r).
\end{aligned}
\eeq
The set $S_{j_1,j_2}$ implicitly depends on $j_0$ as well, due to \eqref{eq:Sdefn}. Recall also that $x_0\neq x_n$, and so $S_{0,n}=\emptyset$. 

The second family of sets is a ``disjointification'' of the first one. We split the path between $x_{j_0}$ and $x_{j_0+1}$, obtaining a ``left piece'' and a ``right piece''. The disjointness is achieved by taking $j_1$ and $j_2$ to be extremal: a path $\ul{x}\in S_{j_1,j_2}$ lies in $S_{j_0,j_1,j_2}'$ iff $j_1$ is the \emph{first} coincidence with the second piece and $j_2$ is the first coincidence with $j_1$.\\

We have the preliminary decomposition
\beq\label{eq:decomp1}
S=\bigcup_{\substack{0\leq j_1\leq j_0\\ j_0<j_2\leq n}} S_{j_1,j_2}=\bigsqcup_{\substack{0\leq j_1\leq j_0\\ j_0<j_2\leq n}} S_{j_0,j_1,j_2}',
\eeq
where $\sqcup$ denotes a disjoint union.\\

Now we depart from the line of argument in \cite{B} and decompose each set $S_{j_0,j_1,j_2}'$ further. We denote
$$
r:=|x_0-x_n|.
$$
Recall that $x_0\neq x_n$ and so $r>0$. The central observation is 

\be{lm}\label{lm:triangle}
Let $\ul{x}\in S_{j_1,j_2}$. Then there exists 
$$
k\in \{0,\ldots,j_1-1\}\cup\{j_2,\ldots,n-1\}
$$
such that
$$
|x_{k}-x_{k+1}|\geq \frac{r}{n}.
$$
\e{lm}

\be{proof}
From $x_{j_1}=x_{j_2}$ and the triangle inequality, we have
$$
r=|x_0-x_n|\leq \sum_{k=0}^{j_1-1}|x_k-x_{k+1}|+\sum_{k=j_2}^{n-1}|x_k-x_{k+1}|.
$$
Hence, at least one of the (at most $n$) terms on the right-hand side must exceed $r/n$. 
\e{proof}

Thanks to Lemma \ref{lm:triangle}, we can decompose the set $S_{j_0,j_1,j_2}'$ further, according to the minimal $k$ satisfying $|x_{k}-x_{k+1}|\geq \frac{r}{n}$. Namely, we define
\beq\label{eq:Sk0defn}
\begin{aligned}
S_{k_0}
:=&\l\{\ul{x}\in (\Z^d)^{n-1}\,:\,
 |x_{k_0}-x_{k_0+1}|\geq \frac{r}{n}
 \textnormal{ and } (*) \textnormal{ holds}\r\}.
\end{aligned}
\eeq
Here $(*)$ encodes the minimality of $k_0$, i.e.,
$$
\begin{aligned}
(*):=
\be{cases}
\max\limits_{0\leq j<k_0}|x_j-x_{j+1}|< \frac{r}{n},\qquad &\textnormal{if } 0\leq k_0<j_1,\\
\max\l\{\max\limits_{0\leq j<j_1}|x_j-x_{j+1}|,\max\limits_{j_2\leq j< k_0}|x_j-x_{j+1}|\r\}< \frac{r}{n},\qquad &\textnormal{if } j_2\leq k_0<n.
\e{cases}
\end{aligned}
$$
Using this, we may refine the preliminary decomposition \eqref{eq:decomp1} as follows:
\beq\label{eq:decomp2}
S=\bigsqcup_{\substack{0\leq j_1\leq j_0\\ j_0<j_2\leq n}} \bigsqcup_{k_0\in ([0,j_1-1]\cup[j_2,n-1])\cap \Z}S_{j_0,j_1,j_2}'\cap S_{k_0},
\eeq
 Note that the union over $k_0$ is indeed disjoint because $k_0$ is chosen minimally.
 
 \subsection{Discarding more reducible paths}
We employ the decomposition \eqref{eq:decomp2} and discard more reducible paths to make a further reduction from Proposition \ref{prop:Skey}.  
 
\be{prop}[Reduction 2]\label{prop:red2}
Let $m\geq0$, $0\leq j_0< n$. Let $j_1,j_2,k_0,k_1,k_2$ be integers such that $0\leq j_1\leq j_0<j_2\leq n$, $k_0\in ([0,j_1-1]\cup[j_2,n-1])\cap \Z$ and $0\leq k_1\leq k_0<k_2\leq n$. Then
\beq\label{eq:propred2}
|f_{S_{j_0,j_1,j_2}'\cap S_{k_0}\cap S_{k_0,k_1,k_2}'}(x_0,x_n)|\leq \eps^3\l(\frac{C}{\eps}\r)^{n} 2^{-m(d-\eps)}\jap{x_0-x_n}^{-2d+\eps}.
\eeq
\e{prop} 
 
We show that Reduction 2 implies Reduction 1 (and hence the main claim).

\be{proof}[Proof of Proposition \ref{prop:Skey} assuming Proposition \ref{prop:red2}]
 We define the set
 \beq\label{eq:Stildedefn}
 \begin{aligned}
 \tilde S:=&\bigsqcup_{\substack{0\leq j_1\leq j_0\\ j_0<j_2\leq n}} \bigsqcup_{k_0\in ([0,j_1-1]\cup[j_2,n-1])\cap \Z}
 \l(S_{j_0,j_1,j_2}'\cap S_{k_0}\r.\\
 &\qquad\qquad\qquad\qquad\l.\cap \setof{\ul{x}}{\{x_0,\ldots,x_{k_0}\}\cap\{x_{k_0+1},\ldots,x_n\}\neq \emptyset}\r).
\end{aligned} 
 \eeq

 We have the following analog of Lemma \ref{lm:red1}:
 
 \be{lm}\label{lm:red2}
 For all $m\geq0$ and $0\leq j_0< n$, we have
 $$
 f_{S}(x_0,x_n)=f_{\tilde S}(x_0,x_n).
 $$
 \e{lm}
This lemma holds because $S\setminus \tilde S$ consists of reducible paths and therefore does not contribute to the sum in \eqref{eq:fdefn}.

Now we recall Definition \eqref{eq:Sjdefn}. We have the finer decomposition
 $$
 \begin{aligned}
 \tilde S=&\bigsqcup_{\substack{0\leq j_1\leq j_0\\ j_0<j_2\leq n}} \bigsqcup_{k_0\in ([0,j_1-1]\cup[j_2,n-1])\cap \Z}\bigcup_{\substack{0\leq k_1\leq k_0\\ k_0<k_2\leq n}} S_{j_0,j_1,j_2}'\cap S_{k_0}\cap S_{k_1,k_2}\\
 =&\bigsqcup_{\substack{0\leq j_1\leq j_0\\ j_0<j_2\leq n}} \bigsqcup_{k_0\in ([0,j_1-1]\cup[j_2,n-1])\cap \Z}\bigsqcup_{\substack{0\leq k_1\leq k_0\\ k_0<k_2\leq n}} S_{j_0,j_1,j_2}'\cap S_{k_0}\cap S_{k_0,k_1,k_2}'.
\end{aligned} 
 $$
Combining Lemma \ref{lm:red2} with this gives
$$
\begin{aligned}
f_{S}(x_0,x_n)=&f_{\tilde S}(x_0,x_n)\\
=&\sum_{\substack{0\leq j_1\leq j_0\\ j_0<j_2\leq n}} \sum_{k_0\in ([0,j_1-1]\cup[j_2,n-1])\cap \Z}\sum_{\substack{0\leq k_1\leq k_0\\ k_0<k_2\leq n}} f_{S_{j_0,j_1,j_2}'\cap S_{k_0}\cap S_{k_0,k_1,k_2}'}(x_0,x_n).
\end{aligned}
$$
Since the total number of summands is bounded by $C^{n}$, \eqref{eq:propred2} implies \eqref{eq:propSkey} and hence Proposition \ref{prop:Skey}.
\e{proof}

In the following section, we give the proof of Proposition \ref{prop:red2}, and this will imply Theorem \ref{thm:main}.

\subsection{Proof of Proposition \ref{prop:red2}}
In this section, we finally see the computation where the gain of $-d$ in the decay exponent comes from. At this point, we have used the randomness sufficiently to our advantage and it is enough to prove a deterministic statement. Indeed, Lemma \ref{lm:transfer} reduces Proposition \ref{prop:red2} to the following estimate
\begin{equation}\label{eqn:red3}
|T^n_{X'}(x_0,x_n)|\leq \eps^3\l(\frac{C}{\eps}\r)^{n} 2^{-m(d-\eps)}\jap{x_0-x_n}^{-2d+\eps},
\end{equation}
where $X':= S_{j_0,j_1,j_2}' \cap S_{k_0}\cap S_{k_0,k_1,k_2}'$. For this, we use a two-step strategy as in \cite{B}. First, we prove the claimed bound for $T^n_{S_{j_1,j_2} \cap S_{k_0}\cap S_{k_1,k_2}}$. Next, we lift this to the bound for $T^n_{X'}$ using Lemma \ref{lm:bourgain2}.

\be{lm}\label{lm:removerandom} Let $T^n_X$ be as in Lemma \ref{lm:transfer} and $m,j_0,j_1,j_2,k_0,k_1,k_2$ be as in Proposition \ref{prop:red2}. Then
$$
|T^n_{S_{j_1,j_2}\cap S_{k_0} \cap S_{k_1,k_2}}(x_0,x_n)|\leq \eps^3\l(\frac{C}{\eps}\r)^{n} 2^{-m(d-\eps)}\jap{x_0-x_n}^{-2d+\eps}.
$$
Moreover, this bound is stable under the choice of functions $\{b_j\}_{1\leq j\leq n}$ with $\norm{b_j}_\infty \leq 1$ and under the replacement
$$
K^j(x,y) \to  e^j_{1}(x)K^j(x,y)e^j_{2}(y)
$$
for any $e^j_{1},e^j_{2}$ satisfying $\norm{e^j_{1}}_{\ell^\infty} \leq 1$ and $\norm{e^j_{2}}_{\ell^\infty} \leq 1$. 
\e{lm}

We first show that \eqref{eqn:red3} follows easily from Lemma \ref{lm:removerandom} via Lemma \ref{lm:bourgain2}. 
\be{proof}[Proof of \eqref{eqn:red3}] 
We use Lemma \ref{lm:bourgain2}. We define the sets
$$
\begin{aligned}
E_1:=\{0,1,\ldots,j_1-1\} \quad \textnormal{ and }&\quad F_1:=\{j_0+1,\ldots,n\},\\
E_2:=\{j_1\} \quad \textnormal{ and }&\quad F_2:=\{j_0+1,\ldots,j_2-1\},\\
E_3:=\{0,1,\ldots,k_1-1\} \quad \textnormal{ and }&\quad F_3:=\{k_0+1,\ldots,n\},\\
E_4:=\{k_1\} \quad \textnormal{ and }&\quad F_4:=\{k_0+1,\ldots,k_2-1\}.
\end{aligned}
$$
These are chosen such that 
$$
X' = X\cap\bigcap_{l=1}^4 \{ \ux\in (\Z^d)^{n-1}: \{ x_u : u\in E_l \} \cap \{ x_v : v\in F_l \} = \emptyset \},
$$
where $X:=S_{j_1,j_2}\cap S_{k_0} \cap S_{k_1,k_2}$. Note that $X$ is a finite set and $\sum_{l=1}^4 |E_l|+|F_l| \leq 4n$. Therefore, \eqref{eqn:red3} follows from Lemma \ref{lm:removerandom} and Lemma \ref{lm:bourgain2}. 
\e{proof}

Finally, we prove Lemma \ref{lm:removerandom}, completing the proof of Theorem \ref{thm:main}. 
\be{proof}[Proof of Lemma \ref{lm:removerandom}] 
To bound 
$\curly{Q}:=T^n_{S_{j_1,j_2}\cap S_{k_0} \cap S_{k_1,k_2}}(x_0,x_n)$, we have to implement the constraints arising from $\ul{x}\in S_{j_1,j_2}\cap S_{k_0} \cap S_{k_1,k_2}$. Two of them are trivial: Writing $\ind_X$ for the indicator function of a set $X$, we have
$$
\ind_{S_{j_1,j_2}\cap S_{k_0} \cap S_{k_1,k_2}}
= \ind_{\setof{\ul{x}}{x_{j_1}=x_{j_2}}} \ind_{\setof{\ul{x}}{x_{k_1}=x_{k_2}}}\ind_{S^{m}_{j_0}\cap S_{k_0}}.
$$
The remaining constraint is that $\ul{x}\in S^{m}_{j_0}\cap S_{k_0}$; these sets were defined in \eqref{eq:Sjupperdefn} and \eqref{eq:Sk0defn}. We write these constraints as intersections of ``local'' constraints, i.e., ones that only depend on a single segment $|x_j-x_{j+1}|$. Namely, we have
$$
\begin{aligned}
S^{m}_{j_0}
=&\bigcap_{0\leq j<j_0} \setof{\ul{x}}{|x_j-x_{j+1}|<2^m} \cap \bigcap_{j_0< j<n} \setof{\ul{x}}{|x_j-x_{j+1}|<2^{m+1}}\\
 &\cap \setof{\ul{x}}{2^m\leq |x_{j_0}-x_{j_0+1}|< 2^{m+1}}.
\end{aligned}
$$
and
$$
\begin{aligned}
S_{k_0}
=&\setof{\ul{x}}{|x_{k_0}-x_{k_0+1}|\geq \frac{r}{n}}\\ 
&\cap\be{cases}
\bigcap\limits_{0\leq j<k_0} \setof{\ul{x}}{|x_j-x_{j+1}|<\frac{r}{n}},\qquad &\textnormal{if } 0\leq k_0<j_1,\\
\bigcap\limits_{\substack{0\leq j<j_1,\\ j_2\leq j<k_0}} \setof{\ul{x}}{|x_j-x_{j+1}|<\frac{r}{n}},\qquad &\textnormal{if } j_2\leq k_0<n.
\e{cases}
\end{aligned}
$$
These expressions imply that 
$$
\ind_{S^{m}_{j_0}\cap S_{k_0}}(\ul{x})=\prod_{j=0}^{n-1} \ind_{I_j}(|x_{j}-x_{j+1}|),
$$
for appropriate intervals $I_j$ (which also depend on $m,n,r,j_0,k_0$).

We have
\beq\label{eq:A2}
\begin{aligned}
\curly{Q}=\sum\limits_{\substack{\ul{x}\in (\Z^d)^{n-1}:\\  x_{j_1}=x_{j_2}, x_{k_1}=x_{k_2}}}  L^1(x_0,x_1) L^2(x_2,x_3) \ldots L^n(x_{n-1},x_n),
\end{aligned}
\eeq
where we introduced the operators $L^j$ with kernels $L^j(x,y)=\ind_{I_j}(|x-y|) K^j(x,y)$.\\

Recall that any path $\ul{x}$ under consideration contains the two ``long'' segments
\beq\label{eq:long}
|x_{j_0}-x_{j_0+1}|\geq R, \qquad
|x_{k_0}-x_{k_0+1}|\geq \frac{r}{n},
\eeq
and $\max_j |x_j-x_{j+1}|\leq 2R$, where we set $R:= 2^m$. From here on, the only data that matters is the collection of relevant times $\{0,n,j_0,j_1,j_2,k_0,k_1,k_2\}$ and their ordering (subject to the usual constraints). By symmetry (we may invert the path), we can assume that $k_0\in [j_2,n-1]\cap \Z$.\\

\dashuline{Case 1:}  Assume that $j_0+1\leq k_1\leq j_2$, so the ``relevant times'' are ordered as follows
\beq\label{eq:1}
0\leq j_1\leq j_0<j_0+1\leq k_1\leq j_2\leq k_0<k_0+1\leq k_2\leq n.
\eeq

We first consider the case where all the relevant times are different, i.e., 
\beq\label{eq:1.1}
0<j_1<j_0<j_0+1<k_1<j_2<k_0<k_0+1<k_2<n
\eeq
and then later indicate necessary modifications for the general case \eqref{eq:1}. Recall that $x_{j_1}=x_{j_2}$ and $x_{k_1}=x_{k_2}$. We denote by $v$ the vector \[ v=(x_{j_0},x_{j_0+1},x_{j_1},x_{k_0},x_{k_0+1},x_{k_1})\in (\Z^d)^6\] and then group the propagators $L^j$ together so that each group corresponds to a time interval in \eqref{eq:1.1}, i.e.,


\beq\label{eq:1written}
\begin{aligned}
\curly{Q}
=&\sum_{v\in (\Z^d)^6}  
\l(\prod_{j=1}^{j_1} L^j\r)(x_0,x_{j_1})
\l(\prod_{j=j_1+1}^{j_0} L^j\r)(x_{j_1},x_{j_0}) 
L^{j_0+1}(x_{j_0},x_{j_0+1})\\ 
&\l(\prod_{j=j_0+2}^{k_1} L^j\r)(x_{j_0+1},x_{k_1})
\l(\prod_{j=k_1+1}^{j_2}L^j\r)(x_{k_1},x_{j_1})
\l(\prod_{j=j_2+1}^{k_0} L^j\r)(x_{j_1},x_{k_0})\\
&L^{k_0+1}(x_{k_0},x_{k_0+1})
\l(\prod_{j=k_0+2}^{k_2} L^j\r)(x_{k_0+1},x_{k_1})
\l(\prod_{j=k_2+1}^{n} L^j\r)(x_{k_1},x_{n}).
\end{aligned}
\eeq

It is important that we retain some of the information contained within the constraints that $|x_{j}-x_{j+1}|\in I_j$ for all $j$. Namely, we need the fact that all the action takes place within some large ball in $\Z^d$. Let $B_\rho(x_0)\subset \Z^d$ be the ball of radius $\rho>0$ around $x_0$. By $\max_j |x_j-x_{j+1}|\leq 2R$ and the triangle inequality, we have 
\beq\label{eq:ball}
x_j\in B_{2nR}(x_0)
\eeq
for all $1\leq j\leq n$. Therefore, we may replace the sum $\sum_{v\in (\Z^d)^6}$ by $\sum_{v\in \curly{B}_n}$ in \eqref{eq:1written}, where $\curly{B}_n:=(B_{2nR}(x_0))^6$. 


We are now in a position to apply Corollary \ref{cor:bourgain}. Note that the operators $L^j(x,y)=\ind_{I_j}(|x-y|) K^j(x,y)$ are equal to $K^j_{I_j}$ from \eqref{eq:KIdefn}. From Corollary \ref{cor:bourgain} and \eqref{eq:long}, we get
\beq\label{eq:1.1bound}
\begin{aligned}
|\curly{Q}|\leq &\l(\frac{C}{\eps}\r)^{n} \eps^9 R^{-d}\l(\frac{r}{n}\r)^{-d}\\
&\times \sum_{v\in \curly{B}_n}
\qmexp{x_0-x_{j_1}}^{-d+\eps}\qmexp{x_{j_1}-x_{j_0}}^{-d+\eps}\qmexp{x_{j_0+1}-x_{k_1}}^{-d+\eps}\qmexp{x_{k_1}-x_{j_1}}^{-d+\eps}
\\&\qquad\quad
\qmexp{x_{j_1}-x_{k_0}}^{-d+\eps}\qmexp{x_{k_0+1}-x_{k_1}}^{-d+\eps}\qmexp{x_{k_1}-x_{n}}^{-d+\eps}.
\end{aligned}
\eeq
We can bound the sums over $x_{j_0},x_{j_0+1},x_{k_0},x_{k_0+1}$ all in the same way. E.g., using that $|x_{j_1}-x_{j_0}|\leq |x_{j_1}-x_{0}|+|x_{0}-x_{j_0}|\leq 4nR$, we have
\beq\label{eq:innerbound}
\sum_{x_{j_0}\in B_{2nR}(x_0)}\qmexp{x_{j_1}-x_{j_0}}^{-d+\eps}
\leq \sum_{y\in B_{4nR}(0)}\qmexp{y}^{-d+\eps}\leq \frac{C^n}{\eps} R^{\eps}.
\eeq
From the bound \eqref{eq:innerbound} and its analogs for $x_{j_0+1},x_{k_0},x_{k_0+1}$, we get
\beq\label{eq:1.1samebound}
\begin{aligned}
|\curly{Q}|\leq &
\l(\frac{C}{\eps}\r)^{n} \eps^5 R^{-d+4\eps}r^{-d}\\
&\times\sum_{x_{j_1},x_{k_1}\in B_{2nR}(x_n)}
\qmexp{x_0-x_{j_1}}^{-d+\eps}\qmexp{x_{k_1}-x_{j_1}}^{-d+\eps}\qmexp{x_{k_1}-x_{n}}^{-d+\eps}.\\
\end{aligned}
\eeq
Using Lemma \ref{lm:sum} twice, we get 
\beq\label{eq:1.1samebound2}
|\curly{Q}| \leq \l(\frac{C}{\eps}\r)^{n} \eps^3 R^{-d+4\eps} r^{-2d+3\eps}. 
\eeq
(We mention that it is possible to replace Lemma \ref{lm:sum} by an elementary observation: Since $|x_0-x_{j_1}|+|x_{k_1}-x_{j_1}|+|x_{k_1}-x_{n}|\geq r$, at least one of these three distances is $\geq r/3$. Implementing this and summing over $x_{j_1},x_{k_1}\in B_{2nR}(x_n)$ gives \eqref{eq:1.1samebound2} with an additional, and irrelevant, $R^{2\eps}$ factor on the right-hand side.)


Next, we turn to the general case \eqref{eq:1}, where some of the relevant times may coincide. We note that  \eqref{eq:1written} is still valid in the general case under the convention that 
\beq\label{eq:simply}
\l(\prod_{j=a+1}^{a} L^j\r) (x_{a}, x_a) \equiv 1.
\eeq

Now we argue why the occurrence of any such coincidences does not change the final bound, \eqref{eq:1.1samebound2}.

Let $A_1, A_2, A_3, A_4$ denote the cases $j_0 = j_1$, $j_0+1 = k_1$, $k_0 = j_2$, $k_0+1 = k_2$, respectively. In addition, let $B_1,B_2,B_3$ denote the cases $j_1=0$, $k_2=n$, $k_1= j_2$, respectively. Then each possible combination of coincidences of the relevant times in \eqref{eq:1} corresponds to a subset of $\{A_1, A_2, A_3, A_4,B_1,B_2,B_3\}$. So far, we considered the case of no coincidences, \eqref{eq:1.1}.

For each occurrence of $A_1,A_2,A_3,A_4, B_1,B_2,B_3$, we have the trivial identity \eqref{eq:simply} instead of having to apply Corollary \ref{cor:bourgain}. Effectively, this amounts to multiplying each summand in \eqref{eq:1.1bound} by $\epsilon^{-1} \delta_{x_a}(x_b)$ for appropriate $a,b$. (Here we denoted by $\delta_x(y)$ the delta function: $\delta_x(y)$ is $1$ if $x=y$ and $0$ otherwise.) For instance, $A_1$ and $B_1$ produce the factors $\epsilon^{-1} \delta_{x_{j_0}}(x_{j_1})$ and $\epsilon^{-1} \delta_{x_{j_1}}(x_{0})$, respectively.

Thus we need to show that the factor $\epsilon^{-1} \delta_{x_a}(x_b)$ leads to the same bound as before, \eqref{eq:1.1samebound2}. 

Consider the case $A_1$, which gives $\epsilon^{-1} \delta_{x_{j_0}}(x_{j_1})$. This is to be compared with how we treated the original expression in \eqref{eq:innerbound}, where the disappearance of the sum over $x_{j_0}$ may alternatively expressed as a bound in terms of $C^n R^\eps\epsilon^{-1} \delta_{x_{j_0}}(x_{0})$. Since $1\leq C^nR^\eps$, we get the same bound, no matter whether $A_1$ occurs or not. The same argument works for $A_2,A_3,A_4$.

To summarize this part, we always get \eqref{eq:1.1samebound} (modified by the appropriate delta functions coming from the cases $B_1, B_2, B_3$), no matter which subset of cases the $A_1,A_2,A_3,A_4$ occurs.

Finally, we come to the cases $B_1, B_2, B_3$. Notice that at most two of them may occur simultaneously because $x_0\neq x_n$. Consider the case where just $B_1$ occurs, i.e., \eqref{eq:1.1samebound} comes with an additional factor $\epsilon^{-1} \delta_{x_{j_1}}(x_{0})$:
$$
\begin{aligned}
|\curly{Q}|\leq &
\l(\frac{C}{\eps}\r)^{n} \eps^4 R^{-d+4\eps}r^{-d}\\
&\times\sum_{x_{j_1},x_{k_1}\in B_{2nR}(x_n)}\delta_{x_{j_1}}(x_{0})
\qmexp{x_0-x_{j_1}}^{-d+\eps}\qmexp{x_{k_1}-x_{j_1}}^{-d+\eps}\qmexp{x_{k_1}-x_{n}}^{-d+\eps}.\\
=&\l(\frac{C}{\eps}\r)^{n} \eps^4 R^{-d+4\eps}r^{-d}\sum_{x_{k_1}\in B_{2nR}(x_n)}\qmexp{x_{k_1}-x_{0}}^{-d+\eps}\qmexp{x_{k_1}-x_{n}}^{-d+\eps}.
\end{aligned}
$$
Lemma \ref{lm:sum} then yields \eqref{eq:1.1samebound2}. Similar considerations imply \eqref{eq:1.1samebound2} for all the other cases as well.\\
 
 \dashuline{Case 2:} Assume that either $0\leq k_1\leq j_0$ or $j_2< k_1\leq k_0$. We may follow exactly the same steps as in Case 1, unless $0\leq k_1 < j_1$, so we assume this in the following. We start by discussing the case, where all the relevant times are different, i.e., 
 $$
 0<k_1<j_1< j_0<j_0+1< j_2< k_0<k_0+1< k_2<n.
 $$

Arguing as in Case 1 and after summing over $x_{j_0},x_{j_{0}+1},x_{k_0},x_{k_0+1}$, we may bound $|\curly{Q}|$ by a slightly different expression (compared to what we got in \eqref{eq:1.1samebound}):
$$
\begin{aligned}
|\curly{Q}|
\leq 
&\l(\frac{C}{\eps}\r)^{n} \eps^5 R^{-d+4\eps}r^{-d}\\
&\times\sum_{x_{j_1},x_{k_1}\in B_{2nR}(x_n)}
\qmexp{x_0-x_{k_1}}^{-d+\eps}\qmexp{x_{k_1}-x_{j_1}}^{-d+\eps}\qmexp{x_{k_1}-x_{n}}^{-d+\eps}\\
\leq 
&\l(\frac{C}{\eps}\r)^{n} \eps^4 R^{-d+5\eps}r^{-d}\sum_{x_{k_1}\in B_{2nR}(x_n)}\qmexp{x_0-x_{k_1}}^{-d+\eps}\qmexp{x_{k_1}-x_{n}}^{-d+\eps}\\
\leq &\l(\frac{C}{\eps}\r)^{n} \eps^3 R^{-d+5\eps}r^{-2d+2\eps}.
\end{aligned}
$$ 

We may also treat the case where some of the relevant times may coincide as in Case 1. Finally, the stability of the bound is a consequence of Corollary \ref{cor:bourgain}. This finishes the proof of Lemma \ref{lm:removerandom}.
 \e{proof}



\section{Proof of Theorem \ref{thm:nlogn} -- partitioning the set of irreducible paths}\label{sec:main2}
In this section, we prove Theorem \ref{thm:nlogn}, which is an immediate consequence of the following decomposition result for $U$.
\begin{prop}\label{prop:decomposition} Let $n\geq 3$ and $T^n_S$ be as in Definition \eqref{eq:tns}. For each $x_0,x_n \in \Z^d$, $x_0\neq x_n$, there is a partition of the set of irreducible paths $U$ into $O(2^n)$ many disjoint subsets $\{U_\alpha' \}_{\alpha\in \mathcal{A}}$ such that
\begin{equation}\label{eqn:setbound}
\max_{\alpha\in \mathcal{A}} |T^n_{U_\alpha'}(x_0,x_n)| \leq C^{n\log n} \epsilon^{3-n} \br{x_0 - x_n}^{-3d+\epsilon}
\end{equation}
for all sufficiently small $\epsilon>0$. 
\end{prop}

We prove Proposition \ref{prop:decomposition} in the following subsections by explicitly constructing the sets $U_\alpha'$.
\subsection{First decomposition}
In the following, we write  $``\{\ux:\ldots\}"$ for $``\{\ul{x}=(x_1,\cdots,x_{n-1}) \in  (\Z^d)^{n-1}:\ldots\}"$. Define the sets
\begin{align*}
V_{i,j}:=& \{ \ux: x_i = x_n , x_j=x_0 \}, \\
V_{i,j}' :=& V_{i,j} \cap \{ \ux: x_0 \notin \{x_{j+1},x_{j+2}, \ldots, x_{n-1} \}\}\\ 
&\quad\cap \{ \ux: x_n \notin \{x_{1},x_2,\ldots,x_{i-1} \}\}.
\end{align*}
In other words, if $\ux\in V_{i,j}'$, then
\[ i = \min\{ l:x_l = x_n \} \;\; \text{ and }\;\; j = \max\{ l: x_l = x_0 \},. \]
The sets $\{ V_{i,j}' \}_{0<i,j<n}$ are disjoint, and we have
$$
\bigsqcup_{0<i<j<n} V_{i,j}' \subset U\subset\bigcup_{0<i,j<n} V_{i,j} = \bigsqcup_{0<i,j<n} V_{i,j}'.
$$
 Therefore, we can decompose 
\begin{equation}\label{eqn:firstdecompo}
U= \l(\bigsqcup_{0<i<j<n} V_{i,j}'\r)\sqcup\l( \bigsqcup_{1<j<i<n-1} U \cap 
V_{i,j}'\r)
\end{equation}
with the observation that $U \cap V_{i,1}' = U \cap V_{n-1,j}' = \emptyset$.

\subsection{A further decomposition}
In this subsection, we further decompose the set 
\[U':=\bigsqcup_{1<j<i<n-1} U \cap 
V_{i,j}'.\] 

{\bf Procedure.} Fix $\ux\in U \cap V_{i,j}'$ for some $j<i$. Since $\ux$ is irreducible, there should exist $i_1<j$ and $j_1 >j$ such that $x_{i_1} = x_{j_1}$. We define
\[ j_1 = \max \{ l : l>j  \text{ and } x_l = x_{i_1} \text{ for some } 0<i_1<j \} \]
and then
\[ i_1 = \min \{ l : 0<l<j \text{ and } x_l = x_{j_1}  \}. \]

Note that, by definition, 
\begin{equation} \label{eqn:max}
\{ x_1,x_2,\ldots,x_{j-1} \} \cap \{ x_{j_1+1}, x_{j_1+2},\ldots,x_{n-1} \} = \emptyset.
\end{equation}

We have the following two alternatives ($j_1 \neq i$ due to the condition imposed on $V_{i,j}'$).
\begin{enumerate}
\item $j_1>i$: we stop with a single pair $(i_1,j_1)$. 
\item $j_1<i$: we continue to choose $(i_2,j_2)$ as follows. 
Since $\ux$ is irreducible, there should be some $i_2<j_1$ and $j_2>j_1$ such that $x_{i_2} = x_{j_2}$. We choose $j_2$ as the maximum of all such $j_2$ and then choose $i_2$ as the minimum of all $l$ such that $x_{l} = x_{j_2}$. From \eqref{eqn:max}, $j < i_2<j_1$.

Having chosen $(i_2,j_2)$, we again have the alternatives:
\begin{enumerate}
\item $j_2>i$: we stop with $(i_1,j_1),(i_2,j_2)$.
\item $j_2<i$: we continue to choose the next pair $(i_3,j_3)$ for some $j_1<i_3 < j_2$ and $j_3> j_2$, following the same procedure. We repeat this procedure until we obtain $(i_1,i_1),\ldots,(i_m,j_m)$ for some $m$ satisfying $j_{m-2}<i_m<j_{m-1}$ and $j_m>i$. We write $m(\ux)$ for this $m$. By a simple counting argument, we see that $m(\ux)\leq (n-3)/2$ for any $\ux \in U'$.
\end{enumerate}
\end{enumerate}

From the Procedure, we may write $U \cap V'_{i,j}$ as a disjoint union. We first define some basic building blocks. For $0<i,j<n$, define
\[\mathcal{S}_{i,j}:=\{ \ux: x_i=x_j \}. \] 
We note that the definition is different from the definition in \eqref{eq:Sjdefn} and \cite{B} -- it does not require a further restriction regarding the dyadic decomposition. 
It is convenient to set $j_0=j$ and $j_{-1}=0$. For $m\geq 1$, define, for a given $(i_m,j_m)$ and fixed $j_{m-2},j_{m-1}$,
\begin{align*}
\mathcal{S}_{i_m,j_m}'(j_{m-2},j_{m-1}) := \mathcal{S}_{i_m,j_m} \setminus
\l(\bigcup_{\substack{ {j_{m-2}<u<j_{m-1}}\\ {j_m<v<n}}} \mathcal{S}_{u,v} \r)\cup\l(
 \bigcup_{j_{m-2} < u < i_m} \mathcal{S}_{u,j_m}\r).
\end{align*}
We shall write $\mathcal{S}_{i_m,j_m}'$ for $\mathcal{S}_{i_m,j_m}'(j_{m-2},j_{m-1})$ for the sake of simplicity. The set $\mathcal{S}_{i_m,j_m}'$ is chosen so that if $(i_m,j_m)$ is selected in the Procedure for $\ux \in U\cap V_{i,j}'$, then $\ux \in \mathcal{S}_{i_m,j_m}'$.

Note that the first step of the Procedure gives
\[ U \cap V'_{i,j} = \bigsqcup_{ \substack{ {0<i_1<j}  \\ {j<j_1<n} }} U \cap V'_{i,j} \cap \mathcal{S}_{i_1,j_1}'. \]
When $j_1<i$, the Procedure decomposes $U\cap V_{i,j}' \cap \mathcal{S}_{i_1,j_1}'$ further; this corresponds to the case 2. We have
$$ U \cap V'_{i,j} \cap \mathcal{S}_{i_1,j_1}' = \bigsqcup_{ \substack{ {j<i_2<j_1}  \\ {j_1<j_2<n} }} U \cap V'_{i,j} \cap \mathcal{S}_{i_1,j_1}' \cap \mathcal{S}_{i_2,j_2}'. $$
Each $U \cap V'_{i,j} \cap \mathcal{S}_{i_1,j_1}' \cap \mathcal{S}_{i_2,j_2}'$ is decomposed further when $j_2<i$; this corresponds to the case 2.(b) in the Procedure.

Repeating this yields the desired decomposition of the set $U'$. To describe this decomposition in a compact way, we set
\begin{align*}
I_m(j,j_{m-2},j_{m-1},i) &:=  \{ (i_m,j_m) : j_{m-2}<i_m<j_{m-1}, i<j_m<n \} \\
I_m'(j,j_{m-2},j_{m-1},i) &:= \{ (i_m,j_m) : j_{m-2}<i_m<j_{m-1}, j_{m-1}<j_m<i \}.
\end{align*}
We shall simply write $I_m$ and $I_m'$ assuming that we work with fixed indices $j,j_{m-2},j_{m-1},i$. Note that if $(i_1,j_1), \ldots, (i_{m(\ux)},j_{m(\ux)})$ are obtained from the Procedure for some $\ux\in U \cap V_{i,j}'$, then 
\begin{align*}
(i_m,j_m) &\in I_m' \text{ for } 1\leq m <m(\ux) \\
(i_{m(\ux)},j_{m(\ux)}) &\in I_{m(\ux)},
\end{align*}
and
\[ \ux \in \mathcal{S}_{i_1,j_1}'\cap \mathcal{S}_{i_2,j_2}' \cap \ldots \cap \mathcal{S}_{i_{m(\ux)},j_{m(\ux)}}'.\]

In conclusion, combined with \eqref{eqn:firstdecompo}, we can write 
\begin{equation}\label{eqn:partition}
\begin{aligned}
U =& \bigsqcup_{0<i<j<n} V_{i,j}'\\
& \sqcup \bigsqcup_{m} \bigsqcup_{1<j<i<n-1}  \bigsqcup_{\substack{ {(i_l,j_l)\in I_l'} \\ {1\leq l< m} }} \bigsqcup_{(i_{m},j_{m})\in I_{m}} V_{i,j}' \cap \mathcal{S}_{i_1,j_1}' \cap \mathcal{S}_{i_2,j_2}' \cap \ldots\cap \mathcal{S}_{i_{m},j_{m}}',
\end{aligned}
\end{equation}
where the union in $m$ is taken over $1\leq m\leq (n-3)/2$.

We write $U=\bigsqcup_{\alpha\in \mathcal{A}} U_\alpha'$ after renaming all disjoint sets involved in \eqref{eqn:partition}. We claim that $\# \mathcal{A} \leq 2^{n-1}$. First note that there are ${n-1} \choose {2}$ sets $V_{i,j}'$, $0<i<j<n$. Moreover, for each $1\leq m\leq (n-3)/2$, the disjoint union $\bigsqcup_{1<j<i<n-1}  \bigsqcup_{\substack{ {(i_l,j_l)\in I_l'} \\ {1\leq l< m} }} \bigsqcup_{(i_{m},j_{m})\in I_{m}},$ involves ${{n-1} \choose {2m+2}}$ sets. Therefore,
\[ \# \mathcal{A} = \sum_{0\leq m \leq (n-3)/2} {{n-1} \choose {2m+2}} \leq 2^{n-1}. \]
For the proof of Proposition \ref{prop:decomposition}, it only remains to prove \eqref{eqn:setbound}.

\subsection{Proof of \eqref{eqn:setbound}}
In this subsection, we prove estimates for each set appearing in the partition \eqref{eqn:partition} using Lemmas \ref{lm:bourgain} and \ref{lm:bourgain2}. Recall that the set $X_k$ is defined by
\[ X_k = \{ \ux\in(\Z^d)^{n-1}: \max_{0\leq j<n }|x_j-x_{j+1}| < 2^k \}. \]
Note that the truncation $S\to S\cap X_k$ amounts to the replacemment $K^j \to K^j_{\mathbf{I}_k}$, where $\mathbf{I}_k=[0,2^k)$.

\begin{lm} \label{lem:vij} Assume that $0<i<j<n$. Then
\[ | T^n_{V_{i,j}'}(x_0,x_n)| \leq \epsilon^3\l(\frac{C}{\eps}\r)^n \br{x_0 - x_n}^{-3d+\epsilon}. \]
\end{lm}
\begin{proof}
First, note that $V_{i,j}' = V_{i,j}\cap A_V$, where
\begin{equation}\label{eq:av}
\begin{split}
A_V &= \{ \ux: \{ x_0\} \cap \{x_{j+1},x_{j+2}, \ldots, x_{n-1} \} = \emptyset \}\\
&\quad \cap \{ \ux: \{x_n\} \cap \{x_{1},x_2,\ldots,x_{i-1} \}=\emptyset\} 
\end{split}
\end{equation}
Therefore, by Lemma \ref{lm:bourgain2}, it is enough to show that 
\[ | T^n_{ V_{i,j} \cap X_k}(x_0,x_n)| \leq \epsilon^3\l(\frac{C}{\eps}\r)^n \br{x_0 - x_n}^{-3d+\epsilon}\]
for all large $k\geq 1$. This is a consequence of the factorization
$$
\begin{aligned} 
T^n_{V_{i,j}\cap X_k }(x_0,x_n)  =& K^1_{\mathbf{I}_k} K^2_{\mathbf{I}_k} \ldots K^{i}_{\mathbf{I}_k}(x_0,x_n) K^{i+1}_{\mathbf{I}_k}\ldots K^j_{\mathbf{I}_k}(x_n,x_0)\\
& K^{j+1}_{\mathbf{I}_k}\ldots K^{n}_{\mathbf{I}_k}(x_0,x_n) 
\end{aligned}
$$
and Corollary \ref{cor:bourgain}.
\end{proof}

\begin{lm} \label{lem:longbound} Let $1 \leq m  \leq (n-3)/2$ and \[ S=V_{i,j} \cap \mathcal{S}_{i_1,j_1} \cap \mathcal{S}_{i_2,j_2}  \cap \ldots \cap  \mathcal{S}_{i_{m},j_{m}} \] for some $(i_l,j_l) \in I_l'$ for $1\leq l<m$ and $(i_m,i_m)\in I_m$. Then 
\[
|T^n_{S}(x_0,x_n)| \leq \epsilon^{3+2m} \l(\frac{C}{\eps}\r)^n \br{x_0-x_n}^{-3d+\epsilon}.
\]
Moreover, we have the bound 
\begin{equation}\label{eqn:longbound}
 |T^n_{S\cap X_k}(x_0,x_n)| \leq \epsilon^{3+2m} \l(\frac{C}{\eps}\r)^n \br{x_0-x_n}^{-3d+\epsilon}
\end{equation}
uniformly in $k\geq 1$.
\end{lm}
\begin{proof}
We only prove the bound for $T^n_S(x_0,x_n)$. The argument for the truncated version \eqref{eqn:longbound} is the same. 

The proof uses an induction on $m$. We start with the base case $m=1$. Note that $0<i_1<j<i<j_1<n$. We may factor $T(S)$ as
\[ T^n_S(x_0,x_n) = \sum_{x_{i_1}} T_1(x_0,x_{i_1})
T_2(x_{i_1},x_{0})
T_3(x_{0},x_{n})
T_4(x_{n},x_{i_1})
T_5(x_{i_1},x_{n}), \]
where $T_1 =  K^1K^2\ldots K^{i_1}$, $T_2=K^{i_1+1} \ldots K^{j}$, $T_3=K^{j+1} \ldots K^{i}$, $T_4 = K^{i+1} \ldots K^{j_1}$, and $T_5=K^{j_1+1} \ldots K^{n}$.

From Lemma \ref{lm:bourgain}, we have
\begin{align}\label{eqn:m1}
|T^n_S(x_0,x_n) | \leq \epsilon^5\l(\frac{C}{\eps}\r)^n \br{x_0-x_n}^{-d+\epsilon} \sum_{x_{i_1}} \br{x_0-x_{i_1}}^{-2(d-\epsilon)}\br{x_{i_1}-x_{n}}^{-2(d-\epsilon)}.
\end{align}
Here and in the following, all sums are over $\Z^d$. The claimed estimate then follows from 
$$
|x_0-x_n|\leq |x_0-x_{i_1}|+|x_{i_1}-x_n|,
$$
which allows to decompose of the summation into two parts: \[\Z^d = \{ x_{i_1} : |x_0-x_{i_1}| \geq |x_0-x_n|/2 \} \cup \{ x_{i_1} : |x_{i_1}-x_n| \geq |x_0-x_n|/2 \}.\]

Next, we shall derive the claimed estimate for $m=2$ from the estimate for $m=1$. Here, $0<i_1<j<i_2<j_1<i<j_2<n$. Following the above argument, we have  
\begin{align*}
|T^n_S(x_0,x_n)|\leq \epsilon^{7} \l(\frac{C}{\eps}\r)^n \sum_{x_{i_1},x_{i_2}} &\br{x_0-x_{i_1}}^{-2(d-\epsilon)} \br{x_0-x_{i_2}}^{-d+\epsilon}\br{x_{i_2}-x_{i_1}}^{-d+\epsilon}\\
&\times \br{x_{i_1}-x_{n}}^{-d+\epsilon} \br{x_{i_2}-x_{n}}^{-2(d-\epsilon)}.
\end{align*}
We first take the sum over $x_{i_2}$ using the Cauchy-Schwarz inequality:
$$
\begin{aligned}
 &\sum_{x_{i_2}}\br{x_0-x_{i_2}}^{-d+\epsilon} \br{x_{i_2}-x_{n}}^{-d+\epsilon}
\br{x_{i_1}-x_{i_2}}^{-d+\epsilon} \br{x_{i_2}-x_{n}}^{-d+\epsilon}\\
 \leq& C \br{x_0-x_{n}}^{-d+\epsilon}   \br{x_{i_1}-x_{n}}^{-d+\epsilon}.
\end{aligned}
$$ 
Then we get the expression \eqref{eqn:m1} up to a multiplicative factor $C\epsilon^2$.

Passing from $m-1$ to $m$ is similar. We omit the details.
\end{proof}

Finally, we pass from \eqref{eqn:longbound} to an estimate for the ``primed" sets. This is the part where we lose a constant factor bounded by $C^{n\log n}$. 
\begin{lm}\label{lem:pass} Let $S'=V_{i,j}' \cap \mathcal{S}_{i_1,j_1}' \cap \mathcal{S}_{i_2,j_2}' \cap \ldots \cap  \mathcal{S}_{i_{m},j_{m}}'$ for some $(i_l,j_l) \in I_l'$ for $1\leq l<m$ and $(i_m,i_m)\in I_m$. Then 
\[|T^n_{S'}(x_0,x_n)| \leq C^{n\log n} \epsilon^{3+2m-n}  |x_0-x_n|^{-3d+\epsilon} \]
for some constant $C>0$.
\end{lm}

To prepare for the proof of Lemma \ref{lem:pass}, we first prove the following weaker estimate.
\begin{equation}\label{eqn:pass0}
|T^n_{S'}(x_0,x_n)| \leq C^{n^2} \epsilon^{3+2m-n}  |x_0-x_n|^{-3d+\epsilon}.
\end{equation}
\begin{proof}[Proof of \eqref{eqn:pass0}]

We first write $\mathcal{S}_{i_m,j_m}'$ as
\begin{align*}
 \mathcal{S}_{i_m,j_m}' = \mathcal{S}_{i_m,j_m} \cap
&\bigcap \{ \ux: \{x_u: j_{m-2}<u<i_{m} \} \cap \{x_{j_m} \} = \emptyset \}  \\
 &\bigcap \{ \ux: \{x_u: j_{m-2}<u<j_{m-1} \} \cap \{x_v : j_m<v<n \} = \emptyset \}.
\end{align*}

Let $S=V_{i_1,j_1}\cap \mathcal{S}_{i_1,j_1} \cap \mathcal{S}_{i_2,j_2}\ldots \cap \mathcal{S}_{i_{m},j_{m}}$. Then we may write 
\[ S' =S \cap A_V \cap A_S^1 \cap A_S^2, \]
where $A_V$ is as in \eqref{eq:av} and
\begin{align*}
A_S^1&= \bigcap_{1\leq l \leq m} 
\{ \ux: \{x_u: j_{l-2}<u<i_{l} \} \cap \{x_{j_l} \} = \emptyset \}\\
A_S^2&= \bigcap_{1\leq l\leq m} \{ \ux: \{x_u: j_{l-2}<u<j_{l-1} \} \cap \{x_v : j_{l}<v<n \} = \emptyset \}. 
\end{align*}
We apply Lemma \ref{lem:longbound} and Lemma \ref{lm:bourgain2}. We count the number of terms $x_j$ needed to define intersections in $A_V$, $A_S^1$, and $A_S^2$. First, $A_V$ involves at most $n-j + i\leq 2n$ terms. In addition, $A_S^1$ involves at most $\sum_{l=1}^m( i_l-j_{l-2})\leq \sum_{l=1}^m (j_l-j_{l-2}) \leq 2n$ terms. Finally, $A_S^2$ involves at most $\sum_{l=1}^m (n-j_l + j_{l-1} - j_{l-2})\leq nm \leq n^2/2.$ 
In total, we lose a factor bounded by $2^{4n+n^2/2} \leq 2^{5n^2}$ in the application of Lemma \ref{lm:bourgain2}. This finishes the proof.
\end{proof}

Next, we indicate how to modify the proof of \eqref{eqn:pass0} to obtain Lemma \ref{lem:pass}. First, recall that the intersection with $A_S^2$ is the only part that we lose a factor larger than $C^n$. We lost a factor of $C^{n^2}$ from the bound
\[ \sum_{l=1}^{m}  |E_l|+ |F_l| \leq n^2/2, \]
where 
\begin{align}\label{eqn:defef}
E_l := (j_{l-2},j_{l-1})\cap \Z  \;\; \text{and} \;\; 
F_{l} := (j_{l}, n) \cap \Z.
\end{align}
We show that we can rewrite $A_S^2$ in a more efficient way, which implies Lemma \ref{lem:pass}.

\begin{lm}\label{lem:effi}
Let $0=j_{-1}<j_0< j_1<\ldots<j_m<n$ be an increasing sequence of integers such that the sets $E_l, F_l$ defined in \eqref{eqn:defef} are non-empty. Then there exist subsets $\{ E_l', F_l' \}_{1\leq l\leq m}$ of $(0,n)\cap \Z$ such that 
\begin{equation}\label{eqn:effi}
\begin{split}
A := & \bigcap_{l=1}^{m} \{ \ux : \{ x_u : u\in E_l \} \cap \{ x_v: v\in F_l \} =\emptyset \}\\ = &\bigcap_{l=1}^{m} \{ \ux : \{ x_u : u\in E_l' \} \cap \{ x_v: v\in F_l' \} = \emptyset \},
\end{split}
\end{equation}
and
\[ \sum_{l=1}^{m}  |E_l'|+ |F_l'| = O(n\log n).\]
\end{lm}
\begin{proof}
We first give an informal discussion. The idea is to choose 
\begin{equation}\label{eqn:middle}
l_0 := \max\{ l : j_l \leq \frac{n}{2}-1 \}
\end{equation}
and then write $A$ as an intersection of three parts
\begin{equation}\label{eqn:three}
A= (A'_{l_0}\cap A_{l_0+1} \cap A_{l_0+2}) \cap \bigcap_{l=1}^{l_0-1} A_l' \cap\bigcap_{l=l_0+3}^{m}  A_l,
\end{equation}
where
\begin{align*}
A_l &:= \{ \ux : \{ x_u : u\in E_l \} \cap \{ x_v: v\in F_l \} = \emptyset \},\\
A_l ' &:= \{ \ux : \{ x_u : u\in E_l \} \cap \{ x_v: v\in F_l\setminus F_{l_0} \} = \emptyset \}, \; \text{ for } l<l_0 \\
 A_{l_0} ' &:= 
\{ \ux : \{ x_u : u\in \cup_{l=1}^{l_0} E_l \} \cap \{ x_v: v\in F_{l_0} \} = \emptyset \}.
\end{align*}
The saving comes from that now $A_l'$ involves $x_v$ for $v\in F_l\setminus F_{l_0}$ when $l<l_0$. We iterate this manipulation to $\bigcap_{l=1}^{l_0-1} A_l'$ and $\bigcap_{l=l_0+3}^{m}  A_l$. 

We turn to a rigorous argument. For a given $0
<j_0< j_1<\ldots<j_m<n$, define
\begin{align*}
\tilde{C}(n;j_0,j_1,\ldots,j_m) &:= \min \sum_{l=1}^{m} (|E_l'| + | F_l'|), 
\end{align*}
where the minimum is taken over all collection of subsets $\{ E_l', F_l' \}_{1\leq l\leq m}$ of $(0,n)\cap \Z$, satisfying \eqref{eqn:effi}. Here the parameter $n$ is associated with the largest element $n-1$ in $F_l$. 
In addition, define 
\[ C(n) := \max \tilde{C}(n;j_0,j_1,\ldots,j_m),\]
where the maximum is taken over all $(j_0,j_1,\ldots,j_m)$ satisfying the assumption of Lemma \ref{lem:effi}. Certainly, $C(n)$ is non-decreasing and $C(n)=O(1)$ when $n=O(1)$.

We claim that $C(n) =O(n\log n)$. Without loss of generality, we may assume that $n$ is a power of $2$. We will show that 
\begin{equation}\label{eqn:induct}
C(n) \leq  3n + 2C(n/2).
\end{equation}
Iterating \eqref{eqn:induct} $k$ times, we get $C(n) \leq 3kn + 2^kC(n/2^k)$, from which we obtain the claim by choosing $k\sim \log n$.

Let $0<j_0<j_1<\ldots<j_m<n$ be given. We need to show that
\begin{equation}\label{eqn:induct2}
\tilde{C}(n;j_0,j_1,\ldots,j_m) \leq 3n + 2C(n/2).
\end{equation}
Let $l_0$ be as in \eqref{eqn:middle} and write $A$ as in \eqref{eqn:three}. Observe that the sets $E_l$ and $F_l\setminus F_{l_0}$ for $l<l_0$ are contained in the set $[1,j_{l_0}]\cap \Z$. 
Therefore, there are subsets $\{E_l',F_l'\}_{1\leq l\leq l_0-1}$ of $[1,j_{l_0}] \cap \Z$ such that 
\[  \bigcap_{l=1}^{l_0-1} A_l' = \bigcap_{l=1}^{l_0-1} \{ \ux : \{ x_u : u\in E_l' \} \cap \{ x_v: v\in F_l' \} = \emptyset \} \]
and 
\[ \sum_{l=1}^{l_0-1}  |E_l'|+ |F_l'| = \tilde{C}(j_{l_0}+1;  j_0,j_1,\ldots,j_{l_0-1} ) \leq C(j_{l_0}+1)\leq C(n/2),\]
since $j_{l_0}+1 \leq n/2$ by the choice of $l_0$.

The situation for $l\geq l_0+3$ is essentially the same as the case for $l<l_0$ since the sets $E_l$ and $F_l$, for $l\geq l_0+3$, are contained in the interval $(j_{l_0+1},n) \cap \Z$ of length less than or equal to $n/2$. Notice that we have translation invariance, i.e., we may work with translated sets of $E_l-j_{l_0+1}$ and $F_l-j_{l_0+1}$ for the purpose of choosing the sets $E_l'$ and $F_l'$. Thanks to this, we may find sets $\{E_l',F_l'\}_{l_0+3 \leq l\leq m}$ such that 
\[  \bigcap_{l=l_0+3}^{m} A_l = \bigcap_{l=l_0+3}^{m} \{ \ux : \{ x_u : u\in E_l' \} \cap \{ x_v: v\in F_l' \} = \emptyset \} \]
and 
\[ \sum_{l=l_0+3}^{m} |E_l'|+ |F_l'| \leq C(n-j_{l_0+1})\leq C(n/2).\]

For the remaining part, $A_{l_0}' \cap A_{l_0+1} \cap A_{l_0+2}$, we just set $E_{l}' = E_l$ and $F_{l}'=F_l$ for $l_0 \leq l\leq l_0+2$ except that $E_{l_0}' := \cup_{l=1}^{l_0}E_l$. 

So far, we have found $\{ E_l',F_l' \}_{1\leq l\leq m}$ satisfying \eqref{eqn:effi} such that
\begin{align*}
 \sum_{1\leq l \leq m} |E_l'| + |F_l'| &\leq 3n + 2C(n/2), 
\end{align*}
which verifies \eqref{eqn:induct2}. This completes the proof.
\end{proof}

\begin{appendix}
\section{Proof of Corollary \ref{cor:GF} on derivatives of the averaged Green's function}
The proof is based on the standard fact that existence of derivatives in Fourier space (which we get from Theorem \ref{thm:main}) can be translated to decay in physical space via integration by parts. For the endpoint case $|\al|=d+1$, we use a variant of the argument which only requires the Fourier transform to be H\"older continuous.

Let $d\geq 2$ and assume that $\alpha$ is a multi-index such that $|\alpha|>2-d$. This condition ensures that the symbol of $\nabla^{\alpha} G$ (and $\nabla^{\alpha} G_\mu$) is integrable on $\T^d$. 
We shall prove the first statement for $d\geq 3$ as the proof of the second statement is identical. 

Fix $0<\epsilon<1$ and let $0<\delta<c\epsilon$, where $c$ is the constant $\tilde{c}_d$ from Theorem \ref{thm:main}. Note that the operator $\mathcal{L}$ is a convolution operator whose symbol is given by \[m(\theta) = (1+\delta \E \sigma) \sum_{j=1}^d 2(1-\cos \theta_j) + \sum_{1\leq j,k\leq d} (e^{-i\theta_j}-1) \widehat{K_{j,k}^\delta}(\theta) (e^{i\theta_k}-1)\]
for $\theta \in \T^d$. By Theorem \ref{thm:main}, we have
\[ \norm{\widehat{K_{j,k}} }_{C^{2d-1,1-\epsilon} (\T^d) }\leq C\delta^2. \]
In particular, we may find $0<c_d \leq c\epsilon$ such that for any $0<\delta<c_d$, we have the lower bound 
\begin{equation}\label{eqn:lower}
|m(\theta)| \geq C |\theta|^2
\end{equation}
for some constant $C>0$ for any $\theta$ in $\T^d$ which we identify with $[-\pi,\pi]^d$. 

Next, let $m^\alpha$ be the symbol of $\nabla^\alpha$, i.e. $ m^\alpha(\theta) = \prod_{j=1}^d (e^{i\theta_j}-1)^{\alpha_j }.$ Since $|m^\alpha(\theta)|\leq \prod_{j=1}^d |\theta_j|^{\alpha_j } \leq |\theta|^{|\alpha|}$, we see that 
\begin{equation}\label{eqn:up}
\abs{\frac{m^\alpha(\theta)}{m(\theta)}} \leq C |\theta|^{|\alpha|-2},
\end{equation}
which is integrable on $\T^d$ provided that $|\alpha|>2-d$.

The kernel $\nabla^\alpha G(x)$ is the Fourier inverse of $m^\alpha(\theta)[m(\theta)]^{-1}$. We estimate $\nabla^\alpha G(x)$ using a dyadic decomposition of $\T^d$ as follows. Let $\vp$ be a smooth even function compactly supported on $[-2,2]$ and $\vp(r) = 1$ for $r\in[-1,1]$. Let $\psi(r) := \vp(r) - \vp(2r)$ and $\psi_l(r) := \psi(2^{l} r)$ for $l\geq 1$ and $\psi_0(r) := 1-\vp(2r)$. Note that $\sum_{l\geq 0} \psi_l(r) = 1$ for any $r\neq 0$. We write $\nabla^\alpha G(x) = \sum_{l\geq 0} \nabla^\alpha G_l(x)$, where we denote by $\nabla^\alpha G_l(x)$ the Fourier inverse of $\psi_l(|\theta|) m^\alpha(\theta)[m(\theta)]^{-1}$. 

Define 
\[ g^\alpha_l(\theta) := \frac{\phi(\theta)m^\alpha(2^{-l}\theta)}{m(2^{-l}\theta)}, \]
where $\phi(\theta) := \psi(|\theta|)$. Then for $l\geq 1$, we may write by a change of variable
\[ \nabla^{\alpha} G_l(x) = 2^{-ld} \int_{\R^d}  g^\alpha_l(\theta) e^{i2^{-l}x \cdot \theta} \frac{d\theta}{(2\pi)^d}. \]

First note that, by \eqref{eqn:up}, $|\nabla^{\alpha} G_l(x)| \leq C 2^{-l(d-2+|\alpha|)}$ for any $x\in \Z^d$. This bound may be improved when $2^{-l}|x|\geq 1$. We claim that when $|x|\geq 2^l$ and $l\geq 1$, we have
\begin{equation}\label{eqn:lbound}
|\nabla^{\alpha} G_l(x)| \leq  \frac{C2^{-l(d-2+|\alpha|)}}{(2^{-l}|x|)^{2d-\epsilon}}.
\end{equation}

Given the estimate \eqref{eqn:lbound}, Corollary \ref{cor:GF} follows quickly. First of all, one can check, using integration by parts, that
\[ |\nabla^{\alpha} G_0(x)| \leq  C(1+|x|)^{-(2d-1)}.\]
Using this bound and \eqref{eqn:lbound}, we get
\begin{equation}\label{eq:GF1}
 |\nabla^{\alpha} G(x)| \leq C \sum_{l\geq0} 2^{-l(d-2+|\alpha|)} \leq C
\end{equation}
for any $x\in \Z^d$, since we assume $|\alpha|>2-d$. Next we assume that $|x|\geq 100$ and study the sum over $2^l > |x|$ and $2^l \leq |x|$ separately. We have 
\begin{equation}\label{eq:GF2}
 \sum_{l\geq 0 :\; 2^l > |x|}  |\nabla^{\alpha} G_l(x)| \leq C\sum_{l\geq 0 :\; 2^l > |x|} 2^{-l(d-2+|\alpha|)} \leq C |x|^{-(d-2+|\alpha|)}. 
\end{equation}
On the other hand, if $|\al|\leq d+1$, we have
\begin{equation}\label{eq:GF3}
\begin{split}
 \sum_{l\geq 0 :\; 2^l \leq  |x|}  |\nabla^{\alpha} G_l(x)| &\leq C|x|^{-(2d-1)} + C\sum_{l\geq 1 :\; 2^l \leq |x|} 2^{l(d + 2 - |\alpha| -\epsilon )} |x|^{-2d+\epsilon}\\ &\leq C |x|^{-(d-2+|\alpha|)}.
 \end{split}
\end{equation}
Observe that \eqref{eq:GF1}, \eqref{eq:GF2} and \eqref{eq:GF3} implies Corollary \ref{cor:GF}.

It remains to verify \eqref{eqn:lbound}. We need the following lemma.
\begin{lm} \label{lm:Holder} For $0<\delta< c_d$ and $l\geq 1$, we have
\[ \norm{g_l^\alpha}_{C^{2d-1,1-\epsilon} (\R^d) }\leq C 2^{-l(|\alpha|-2)}.  \]
\end{lm}
\begin{proof}
When $\theta \in \supp \phi$, $|m(2^{-l}\theta)|$ is comparable to $2^{-2l}$ as $|\theta|\sim1$. In addition, we have the estimates
\begin{align*}
|[(2^{-l}\partial)^{\beta} m](2^{-l}\theta)| &\leq C_\beta   2^{-2l} \\
|[(2^{-l}\partial)^{\beta} m^\alpha ](2^{-l}\theta)| &\leq C_{\beta,\alpha} 2^{-l|\alpha|}
\end{align*}
for all multi-index $\beta$ with $|\beta|\leq 2d-1$. From these estimates, if follows that 
\[ \norm{g_l^\alpha}_{C^{2d-1} (\R^d) }\leq C 2^{-l(|\alpha|-2)}. \]

For the H\"{o}lder estimate, we note that when $|\beta| = 2d-1$, we may write 
\[ \partial^\beta g_l^\alpha(\theta) = 
\frac{\chi(\theta) 2^{-l(2d-1)} \partial^\beta m (2^{-l}\theta) m^{\alpha} (2^{-l}\theta)}{m(2^{-l}\theta)^2} + R(\theta), \]
where $\norm{R}_{C^1} \leq C 2^{-l(|\alpha|-2)}$. Thus, it remains to show that the $C^{0,1-\epsilon}$ norm of the first term is $O(2^{-l(|\alpha|-2)})$. This again reduces to quantify the $C^{0,1-\epsilon}$ norm of the functions resulting from replacing $\partial ^\beta m(2^{-l}\theta)$ in the first term by
\[ 
(e^{-i2^{-l}\theta_j} -1)(e^{i2^{-l}\theta_k} -1) \partial^\beta \widehat{K^\delta_{j,k}}(2^{-l}\theta) \]
for each $1\leq j,k\leq d$. One can verify that $C^{0,1-\epsilon}$ norm of the resulting functions are $O(2^{-l(|\alpha|-2)})$.
\end{proof}
Finally, we may deduce \eqref{eqn:lbound} from Lemma \ref{lm:Holder} by a standard argument. Let $|x| \geq 2^l$. Without loss of generality, we may assume that $|x_1| = \max_j |x_j|$, hence $|x_1| \sim |x|$. Using integration by parts, we see that 
\[ 
\nabla^{\alpha} G_l(x) = \frac{C2^{-ld}}{(2^{-l}x_1)^{2d-1}} \int (\partial_1)^{2d-1} g^\alpha_l(\theta) e^{i2^{-l}x\cdot \theta} d\theta. \]
After the change of variable $\theta_1 \to \theta_1 + \frac{\pi}{2^{-l}x_1}$ in the integral, we also see that 
\[ 
\nabla^{\alpha} G_l(x) = - \frac{C2^{-ld}}{(2^{-l}x_1)^{2d-1}} \int (\partial_1)^{2d-1} g^\alpha_l\l(\theta_1+\frac{\pi}{2^{-l}x_1}, \theta' \r) e^{i2^{-l}x\cdot \theta} d\theta, \]
where we write $\theta' = (\theta_2,\cdots,\theta_d)$. Estimating the average of these expressions for $\nabla^{\alpha} G_l(x)$ using Lemma \ref{lm:Holder}, we obtain \eqref{eqn:lbound} which finishes the proof.

\section{Proof of Corollary \ref{cor:GF2} on averaged solutions}
We have $f\in H^{-1}(\Z^d)$ by \eqref{eqn:HLS} and $u_\om=L_\om ^{-1} f$ is the unique solution in $H^1(\Z^d)$ to the equation $L_\om u_\om = f$. 

By the definition of $\curly{L}$, we have $\mathbb{E} [u_\om ] = \curly{L}^{-1} f$. In addition, we have
\begin{equation}\label{eqn:app}
\curly{L} (G*f) = (\curly{L} G)*f = \delta_0 * f = f
\end{equation}
which yields $ \curly{L} ^{-1} f = G*f$. We need to verify the first equality of \eqref{eqn:app}, which is trivial when $f$ is compactly supported. For general $f\in \ell^{p_d}(\Z^d)$, it suffices to show that
\begin{equation}\label{eqn:inter}
 \nabla_i^* K^\delta_{i,j} \nabla_{j} (G*f) =( \nabla_i^* K^\delta_{i,j} \nabla_{j} G)*f. 
\end{equation} 
To see this, first note that the sum defining the convolution $G*f$ converges absolutely since $G\in \ell^{q_d}(\Z^d)$ and $f\in \ell^{p_d}(\Z^d)$ and $\frac{1}{p_d}+\frac{1}{q_d} = 1$. This shows that $\nabla_j (G*f) = (\nabla_j G)*f$ with $\nabla_j G \in \ell^{q_d}(\Z^d)$. In fact, $\nabla_j G\in \ell^2(\Z^d)$ since $G\in H^1(\Z^d)$, but we do not use this fact here. Moreover, the kernel of $K_{i,j}$ belongs to $\ell^1(\Z^d)$ and we have $K_{i,j} [( \nabla_j G) * f] = (K_{i,j}  \nabla_j G) * f $ by Fubini's theorem and $K_{i,j}  \nabla_j G \in \ell^{q_d}(\Z^d)$. The argument for $\nabla_i^*$ is the same and this establishes \eqref{eqn:inter}. 

The pointwise estimate is a direct consequence of Corollary \ref{cor:GF}.
\qed

\section{Proof of the deterministic bound in Lemma \ref{lm:bourgain}}
We closely follow \cite{B} and provide some details. Recall that 
\[ T^n(x_0,x_n) = \sum_{\ux } K^1(x_0,x_1)K^2(x_1,x_2)\ldots K^n(x_{n-1},x_n), \]
where $\ux = (x_1,x_2,\ldots,x_{n-1})\in (\Z^d)^{n-1}$. When $x_0\neq x_n$, we may write
\[ \sum_{\ux} = \sum_{m\geq 0} \sum_{\ux:\; 2^m \leq \max_j{|x_j-x_{j+1}|} < 2^{m+1}} = \sum_{j_0=0}^{n-1}  \sum_{m\geq 0} \sum_{\ux\in S_{j_0}^{m}}, \]
where $S_{j_0}^{m}$ is defined in \eqref{eq:Sjupperdefn}. When $x_0 =x_n$, this yields a decomposition for the sum $\sum_{\ux}$ except for $\ux = (x_0,\cdots,x_0)$ for which we may invoke the bound $|K^1(x_0,x_0)\cdots K^n(x_0,x_0)|\leq A^n$. 

Let $\mathbf{I}_m = [0,2^m)$. Observe that
\begin{equation}\label{eqn:factor}
\begin{split}
 &\sum_{\ux \in S^m_{j_0} } K^1(x_0,x_1)\ldots K^n(x_{n-1},x_n) \\
 &= \sum_{\ux\in (\Z^d)^{n-1}} \prod_{j=1}^{j_0} K^j_{\mathbf{I}_m}(x_{j-1},x_j) K^{j_0+1}_{\mathbf{I}_{m+1}\setminus \mathbf{I}_{m}}(x_{j_0}, x_{j_0+1}) \prod_{j=j_0+2}^{n} K^j_{\mathbf{I}_{m+1}}(x_{j-1},x_j) \\
 &= \sum_{x_{j_0}, x_{j_0+1} } T^{j_0}_{\mathbf{I}_m}(x_0,x_{j_0}) K^{j_0+1}_{\mathbf{I}_{m+1}\setminus \mathbf{I}_{m} }(x_{j_0}, x_{j_0+1}) \widetilde{T}^{j_0}_{\mathbf{I}_{m+1}} (x_{j_0+1},x_n),
 \end{split}
\end{equation}
where $T^{j_0}_{\mathbf{I}_m} :=  \prod_{j=1}^{j_0} K^j_{\mathbf{I}_m}$ and $\widetilde{T}^{j_0}_{\mathbf{I}_{m+1}}  := \prod_{j=j_0+2}^{n} K^j_{\mathbf{I}_{m+1}}$. Here, the sum over $x_{j_0}$ is in fact a finite sum; $|x_0 - x_{j_0}| \leq |x_0- x_1| + \ldots+|x_{j_0-1}-x_{j_0}| \leq n 2^{m+1}.$ Similarly, the sum over $x_{j_0+1}$ is a finite sum over $|x_{j_0+1}-x_n|\leq n2^{m+1}$. From this, H\"{o}lder's inequality, and Assumption (i),
\begin{align*}
&|\eqref{eqn:factor}|\\
 &\leq A2^{-md} \sum_{x_{j_0}} |T^{j_0}_{\mathbf{I}_m} (x_0,x_{j_0})| \sum_{x_{j_0+1}} |\widetilde{T}^{j_0}_{\mathbf{I}_{m+1}}  (x_{j_0+1},x_n)| \\
& \leq C A2^{-md} (n2^m)^{2d(p-1)/p} \Big(\sum_{x_{j_0}}|T^{j_0}_{\mathbf{I}_m} (x_0,x_{j_0})|^p\Big)^{1/p} \Big(\sum_{x_{j_0+1}}|\widetilde{T}^{j_0}_{\mathbf{I}_{m+1}} (x_{j_0+1},x_n)|^p \Big)^{1/p}
\end{align*}
for $p>1$ selected by $2d(p-1)=\epsilon$. 

Let $\delta_y$ be the delta function on $\Z^d$; $\delta_y(x)$ is equal to $1$ if $x=y$ and $0$ otherwise. Note that the product of two $\ell^p$ sums in the last inequality is bounded by
\[ \norm{(T^{j_0}_{\mathbf{I}_m} )^* \delta_{x_0}}_{\ell^p(\Z^d)}  \norm{\widetilde{T}^{j_0}_{\mathbf{I}_{m+1}} \delta_{x_n}}_{\ell^p(\Z^d)}, \]
which is bounded by $[A/(p-1)]^{j_0} [A/(p-1)]^{n-j_0-1} =\l(\frac{2dA}{\epsilon}\r)^{n-1}$ by Assumption (2). Therefore, we get 
\begin{align}\label{eqn:bd}
|\eqref{eqn:factor}|  \leq C n^\epsilon \l(\frac{2dA}{\epsilon}\r)^{n} \epsilon 2^{-m(d-\epsilon)} .
\end{align}

It only remains to sum \eqref{eqn:bd} over $m\geq 0$ and $j_0$. For this, we distinguish the cases $|x_0-x_n|\geq 2n$ and $|x_0-x_n|< 2n$. For the first case, the sum over $m$ is restricted to 
\[ 2^m \geq \frac{|x_0-x_n|}{ 2n}, \]
which follows from, given that $\max_j|x_j-x_{j+1}| < 2^{m+1}$,  
\[ |x_0-x_n| \leq \sum_{j=0}^{n-1} |x_j-x_{j+1}| \leq n 2^{m+1}. \]
In the first case, therefore, summing \eqref{eqn:bd} over $m$ and $j_0$ yields
\begin{align*}
\begin{split}
 |T^n(x_0,x_n)| &\leq C  n^{1+\epsilon}\l(\frac{2dA}{\epsilon}\r)^{n} \epsilon \sum_{m: \; 2^m \geq |x_0-x_n|/ (2n)}  2^{-m(d-\epsilon)} \\
 & \leq (C_d A/\epsilon)^n \epsilon \br{x_0-x_n}^{-(d-\epsilon)}. 
 \end{split}
\end{align*}

When $|x_0-x_n|< 2n$, we sum \eqref{eqn:bd} over $m\geq 0$ and $j_0$. Then we get 
\begin{align*}
 |T^n(x_0,x_n)| \leq C  n^{1+\epsilon} \l(\frac{2dA}{\epsilon}\r)^{n} \epsilon 
\end{align*}
which completes the proof since $1 \leq Cn \br{x_0-x_n}^{-1}$.

\section{Proof of Lemma \ref{lm:bourgain2} on disjointification}
Since $S'=\emptyset$ when $E_l \cap F_l \neq \emptyset$ for some $l$, we may assume that $E_l \cap F_l = \emptyset$ for all $1\leq l\leq m$.

We closely follow the argument given in Section 4 of \cite{B}.  We introduce an additional set of variables (``Steinhaus system'') on the torus $\mathbb T=\R/2\pi \Z$
$$
\ol{\theta}:=(\ol{\theta}^1, \ol{\theta}^2, \ldots, \ol{\theta}^m ), \text{ where } \ol{\theta}^l := 
\setof{\theta^l_x\in \mathbb T}{x\in\Z^d}.
$$
We use these variables to define the complex-valued functions $e_j:\Z^d\to \C$,
$$
\begin{aligned}
e_j(x,\ol{\theta}):=&\prod_{l=1}^m \exp\l(i \nu^l_j \theta^l_{x}\r),\\
\textnormal{where }\quad \nu^l_j:=&\be{cases}
1,\qquad \;\textnormal{ if }j\in E_l,\\
-1,\qquad \textnormal{if }j\in F_l,\\
0,\qquad \;\textnormal{ otherwise}.
\e{cases}
\end{aligned}
$$
Note that $\|e_j(\cdot,\ol{\theta})\|_\it \leq 1$ for all $\ol{\theta}$.

Assume first that the set $S$ is finite. Define 
\begin{align*}
&\tilde T^n_S(x_0,x_n,\ol{\theta}):= \sum_{\ul{x}\in S} \tilde K^1(x_0,x_1)\tilde K^2(x_1,x_2)\ldots \tilde K^n(x_{n-1},x_n) \\
&= \sum_{\ul{x}\in S} K^1(x_0,x_1) \ldots K^n(x_{n-1},x_n) \prod_{l=1}^{m} \exp \l( i \l(\sum_{j\in E_l} \theta^l_{x_j} - \sum_{k\in F_l} \theta^l_{x_k} \r) \r)
\end{align*}
where we introduced the operators $\tilde K^j(x,y):=K^j(x,y)e_j(y,\ol{\theta})$ for $2\leq j\leq n$ and $\tilde K^1(x,y):=e_0(x,\ol{\theta})K^1(x,y)e_1(y,\ol{\theta})$. It is important to observe that, by the assumption, we have 
\beq\label{eq:Q1bound}
|\tilde T^n_S(x_0,x_n,\ol{\theta})|\leq M(x_0,x_n)
\eeq
for all $\ol{\theta}$.

The next step is to average the bound \eqref{eq:Q1bound} over the variables $\ol{\theta}$ with respect to specific probability measures to be chosen. Define the set
$$
\begin{aligned}
 \Z^d_S :=& \{ x_0,x_n\}\\
 & \cup \{ x\in \Z^d: x=x_j \text{ for some } (x_1,\ldots,x_{n-1})\in S \text{  and  }  1\leq j\leq n-1\} 
 \end{aligned}
 $$
which is finite since $S$ is finite by assumption. For each $-1<t<1$, let $P_t(\theta)$ be the Poisson kernel of the unit disk
\[
P_t(\theta)=\sum_{n=-\it}^\it t^{|n|}e^{in\theta}.
\]
Note that $P_t(\theta) \frac{d\theta}{2\pi}$ is a probability measure on $\T$. For each $1\leq l\leq m$ and $|t|<1$, consider the product measure $\d \mu_t^l$ on $\T^{\Z^d_S}= \prod_{x\in \Z^d_S} \T$ given by
\[ \d \mu_t^l (\ol{\theta}^l) := \prod_{x\in \Z^d_S} P_t(\theta^l_x) \frac{d\theta^l_x}{2\pi}.\] 

We first average \eqref{eqn:pointbd} over the probability space $\T^{\Z^d_S}$ equipped with the measure $\d\mu_{t}^1$. On the one hand, since \eqref{eq:Q1bound} holds pointwise in $\ol{\theta}$, we have 
\begin{equation}\label{eqn:pointbd}
\l|\int_{\mathbb T^{\Z^d_S}} \tilde T^n_S(x_0,x_n,\ol{\theta})
\d\mu_{t}^1\l(\ol{\theta}^1 \r)\r| \leq M(x_0,x_n)
\end{equation}
for any $(\ol{\theta}^2, \ldots, \ol{\theta}^m)$ and $|t|<1$. On the other hand, we may write the integral above as
$$
\begin{aligned}
& \sum_{\ul{x}\in S} K^1(x_0,x_1)K^2(x_1,x_2)\ldots K^n(x_{n-1},x_n) \prod_{l=2}^{m} \exp \l( i \l(\sum_{j\in E_l} \theta^l_{x_j} - \sum_{k\in F_l} \theta^l_{x_k} \r) \r) \\
&\times
\int_{\T^{\Z^d_S}} \exp \l( i \l(\sum_{j\in E_1} \theta^1_{x_j} - \sum_{k\in F_1} \theta^1_{x_k} \r) \r) \prod_{x\in \Z^d_S} P_{t}(\theta^1_x) \frac{\d\theta_{x}^1}{2\pi}.
\end{aligned}
$$
As was observed in \cite{B}, the integral in the above line is equal to $ t^{w_{(x_0,x_1,\ldots,x_n)}}$, where 
\[ w_{(x_0,\ldots,x_n)} = \sum_{x \in \Z^d_S} \bigl| { | \{ j\in E_1: x_j=x \}| - | \{ k\in  F_1: x_k=x \}| }\bigr| \leq |E_1| + |F_1|. \]
Moreover, $w_{(x_0,\ldots,x_n)}= |E_1|+|F_1|$ if and only if $\ul{x} \in S_1 $, 
where 
\[ S_1 := S\cap \{\ul{x}: \{x_j: j\in E_1 \} \cap \{x_k:k\in F_1 \} = \emptyset \}. \]
Therefore, we may write $\int_{\mathbb T^{\Z^d_S}} \tilde T^n_S(x_0,x_n,\ol{\theta})
\d\mu_{t}^1(\ol{\theta}^1 )$ as a polynomial 
\begin{equation} \label{eqn:poly}
f(t)= a_D t^D +a_{D-1}t^{D-1} + \ldots + a_0,
\end{equation}
where $D= |E_1| + |F_1|$ and 
$$
\begin{aligned} 
a_D =&  
  \sum_{\ul{x}\in S_1} K^1(x_0,x_1)K^2(x_1,x_2)\ldots K^n(x_{n-1},x_n) \\
   &\prod_{l=2}^{m}\exp \l( i \l(\sum_{j\in E_l} \theta^l_{x_j} - \sum_{k\in F_l} \theta^l_{x_k} \r) \r). 
\end{aligned}
$$

At this point, we recall a special case of the Markov brothers' inequality.
\begin{lm}
Let $f(t)$ be a polynomial as in \eqref{eqn:poly}. Then we have
\[ |a_D| \leq 2^{D-1} \max_{ -1\leq t \leq 1}  |f(t)|. \]
\end{lm}
Combined with \eqref{eqn:pointbd}, we get, with $D=  |E_1| + |F_1|$,
\[ |a_D| \leq 2^{D-1} M(x_0,x_n) \]
for any $(\ol{\theta}^2, \ldots, \ol{\theta}^m)$.
What comes next is a similar averaging argument for the top coefficient $a_D$ over the measure $\d \mu_t^2$, which yields
 \[
 \begin{split}
 & \left|  \sum_{\ul{x}\in S_2} K^1(x_0,x_1)K^2(x_1,x_2)\ldots K^n(x_{n-1},x_n) \prod_{l=3}^{m} \exp \l( i \l(\sum_{j\in E_l} \theta^l_{x_j} - \sum_{k\in F_l} \theta^l_{x_k} \r) \r) \r| \\
& \leq \l(\prod_{1\leq  l \leq 2} 2^{|E_l|+|F_l|-1} \r) M(x_0,x_n).
\end{split}
 \]
for any $(\ol{\theta}^3, \ldots, \ol{\theta}^m)$, where 
\[ S_2 := S_1 \cap \{\ul{x}: \{x_j: j\in E_2 \} \cap \{x_k:k\in F_2 \} = \emptyset \}. \]
A successive averaging over $\d\mu_{t}^3, \ldots, \d\mu_{t}^m$ finishes the proof.

Next, assume that $S$ is not a finite set. By dominated convergence and the a priori bound \eqref{eq:log}, we have $T^n_{S'} (x_0,x_n)= \lim_{k\to \infty} T^n_{ S' \cap X_k}(x_0,x_n)$. Therefore, it is sufficient to show that 
\[
 |T^n_{S' \cap X_k}(x_0,x_n)| \leq 2^{\sum_{1\leq l\leq m} |E_l|+|F_l|} M(x_0,x_n)
\]
for all large $k\geq 1$, which follows from applying the result for the finite set to $S\cap X_k$.
\qed

\end{appendix}

\end{document}